\documentclass[11pt, reqno]{amsart}

\usepackage[text={6.5in,9.5in},centering]{geometry}
\usepackage{amsmath,amsfonts,amsthm,amssymb,color,hyperref}
\usepackage{enumitem}
\usepackage{mathrsfs}

\usepackage[T1]{fontenc}

\usepackage{graphicx}
\usepackage{subfigure} 

\makeatletter
     \def\section{\@startsection{section}{1}%
     \z@{.7\linespacing\@plus\linespacing}{.5\linespacing}%
     {\bfseries
     \centering
     }}
     \def\@secnumfont{\bfseries}
     \makeatother

\usepackage{graphicx}

\newcommand{\R}{\mathbb R}
\newcommand{\RR}{\mathbb R}

\newcommand{\E}{\mathbb E}

\newcommand{\1}{ 1}

\newcommand{\HH}{\mathfrak H}

\newtheorem{theorem}{Theorem}[section]
\newtheorem{lemma}[theorem]{Lemma}
\newtheorem{assumption}[theorem]{Assumption}
\newtheorem{proposition}[theorem]{Proposition}

\theoremstyle{definition}

\theoremstyle{remark}
\newtheorem{remark}{Remark}
\newtheorem{example}{Example}
\numberwithin{equation}{section}
\begin{document}
\title[Approximation for stochastic fractional heat equation]{Quantitative normal approximations for the stochastic fractional heat equation}

\author[O. Assaad]{Obayda Assaad}
\address{Universit\'e de Lille 1, UFR Math\'ematiques, France}
\email{obayda.assaad@univ-lille.fr}

\author[D. Nualart]{David Nualart} \thanks {D. Nualart is supported by NSF Grant DMS 1811181.}
\address{University of Kansas, Department of Mathematics, USA}
\email{nualart@ku.edu}

\author[C.A. Tudor]{Ciprian A. Tudor}
\address{Universit\'e de Lille 1, UFR Math\'ematiques, France}
\email{ciprian.tudor@univ-lille.fr}

\author[L. Viitasaari]{Lauri Viitasaari}
\address{Aalto University School of Business, Department of Information and Service Management, Finland}
\email{lauri.viitasaari@iki.fi}

\begin{abstract}
In this article we present a {\it quantitative} central limit theorem for  the     stochastic fractional heat  equation driven by a a general Gaussian multiplicative noise, including the cases of space-time white noise and the white-colored noise with spatial covariance given by the Riesz kernel or a bounded integrable function.     We   show that the   spatial average over a ball of radius $R$ converges, as $R$ tends to infinity, after suitable renormalization, towards a Gaussian limit in the total variation distance. We also provide a functional central limit theorem. As such, we extend recently proved similar results for  stochastic heat equation to the case of the fractional Laplacian and to the case of general noise. 
\end{abstract}

\maketitle



\allowdisplaybreaks

\section{Introduction}
In this article we consider the stochastic fractional heat equation
\begin{equation}
\label{1}
\frac{\partial u}{\partial t} (t,x)= - (-\Delta)^{\frac{\alpha}{2}} u(t,x) + \sigma (u(t,x)) \dot{W} (t,x), \hskip0.5cm t\geq 0, x \in \mathbb{R}^d
\end{equation}
with initial condition $u(0, x)\equiv 1$. Here $\sigma$ is assumed to be a Lipschitz continuous function with the property $\sigma(1)\neq 0$ and $- (-\Delta)^{\frac{\alpha}{2}}$ is the fractional Laplace operator.

Fractional Laplace operator can be viewed as a generalization of spatial derivatives and classical Sobolev spaces into fractional order derivatives and fractional Sobolev spaces, and together with the  associated equations it has numerous applications in different fields including   fluid dynamics, quantum mechanics, and finance to simply name a few. For detailed discussions and different equivalent formal definitions, see \cite{Garofalo} and the references therein.

In the present article we provide a general existence and uniqueness result to  equation \eqref{1} that covers many different choices of the (Gaussian) random perturbation $\dot{W}$. Moreover, we provide quantitative limit theorems in a general context. These results cover three different important situations: when $\dot{W}$ is a standard space-time white noise, when $\dot{W}$ is a white-colored noise, i.e. a Gaussian field that behaves as a Wiener process in time and it has a non-trivial spatial covariance given by the Riesz kernel of order $\beta <\min(\alpha,d)$, and when $\dot{W}$ is a white-colored noise with spatial covariance given by an integrable and bounded function $\gamma$. 

Our results continue the line of research initiated in \cite{HNV18} and \cite{HNVZ19} where a similar problem for the   stochastic heat equation on $\R$ (or $\R^d$, respectively) driven by a space-time white noise (or spatial covariance given by the Riesz kernel, respectively) was considered. As such, we extend the results presented in \cite{HNV18}  and \cite{HNVZ19} as the main theorems of \cite{HNV18} and \cite{HNVZ19}  can be recovered from ours by simply plugging in $\alpha=2$. Proof-wise our methods are similar to those of these two references. However, we stress that in our case we do not have fine properties of the heat kernel at our disposal, and hence one has to be more careful in the computations. In particular, our main contribution is the bound for the norm of the Malliavin derivative (cf. Proposition \ref{prop:malliavin-bound}) that differs from the classical Laplacian case. Moreover, we provide a general approach how such bounds can be achieved, based on the boundedness properties of the convolution operator with the spatial covariance $\gamma$ (see Proposition \ref{p1}) together with the semigroup property and some integrability of the Green kernel. We also remark that, while the existence of mild solutions to (\ref{1}) in the case of the space-time white noise is a known fact (see \cite{DD}), to the best of our knowledge there are no results that would provide the existence in our generalised framework. In the present article, we provide such a result under general, so-called fractional Dalang condition. 

On a related literature, we also mention \cite{DNZ18} studying the case of stochastic wave equation on $\R^d$. In this article, the driving noise was assumed to be Gaussian multiplicative noise that is white in time and colored in space such that the correlation in the space variable is described by the Riesz kernel. As such, our results complements the above mentioned works studying the stochastic heat and wave equation.
 
The rest of the paper is organised as follows. In Section \ref{sec:results} we describe and discuss our main results. In particular, we provide the existence and uniqueness result for the solution, and provide quantitative central limit theorems for the spatial average in the mentioned particular cases. In Section \ref{sec:prel} we recall some preliminaries, including some basic facts on Stein's method and Malliavin calculus that are used to prove our results, together with some basic facts on the Green kernel related to the fractional heat equation, and a key inequality proved in Proposition \ref{p1}. Proofs of our main results are provided in Section \ref{sec:existence-proof} and Section \ref{sec:clt-proofs}.

\section{Main results}
\label{sec:results}
In this section we introduce and discuss our main results concerning  equation \eqref{1}.
Throughout the article, we assume that $\dot{W}$ is a centered  Gaussian noise with a covariance given by
\begin{equation}
\label{eq:cov-general}
\E [\dot{W}(t,x)\dot{W}(s,y)] = \delta_0(t-s)\gamma(x-y),
\end{equation}
where $\delta_0$ denotes the Dirac delta function and $\gamma$ is a nonnegative and nonnegative definite symmetric measure. The spectral measure $\widehat{\gamma}(d\xi)$ is defined through the Fourier transform of the measure $\gamma$:
$$
\widehat{\gamma}(d\xi)= \left(\mathcal{F}\gamma\right)(d\xi) := \int_{\mathbb{R}^d} e^{-i\langle \xi,y\rangle}d\gamma(y).
$$
 The existence of the solution to \eqref{1} is guaranteed if a fractional version \eqref{eq:dalang} of  Dalang's condition is satisfied. In particular, this is the case on all examples mentioned in the introduction. 

We next introduce the Green kernel (or fundamental solution) associated to the operator $- (-\Delta)^{\frac{\alpha}{2}}$, where $\alpha \in (0,2]$. This kernel, denoted in the sequel by $G_{\alpha}$, is defined via its Fourier transform 
\begin{equation}
\label{FG}
\left( \mathcal{F} G_{\alpha}(t, \cdot)\right) (\xi) =e ^ {-t\vert \xi\vert ^ {\alpha} }, \hskip0.5cm \xi \in \mathbb{R}^d, t\geq 0
\end{equation}
for $\alpha>0$ (here and in the sequel, $\vert \cdot \vert$ denotes the Euclidean norm). While explicit formulas for $G_\alpha(t,x)$ are known only in the special cases $\alpha = 1$ (the Poisson kernel) and $\alpha=2$ (the heat kernel), the kernel $G_\alpha(t,x)$ admits many desirable properties. Some of them that are suitable for our purposes are recorded in Section \ref{sec:prel}. 

Similarly to the classical stochastic heat equation case, the solution to the stochastic equation \eqref{1} can be expressed in terms of $G_{\alpha}$. That is, the mild solution is a measurable random field $\left( u(t,x), t\geq 0, x \in \mathbb{R}^d \right)$ which satisfies 
\begin{equation}
\label{mild}
u(t,x) =1+ \int_{0} ^ {t} \int_{\mathbb{R}^d} G_{\alpha} (t-s, x-y) \sigma (u(s, y)) W (ds, dy),
\end{equation}
where the stochastic integral is understood in the Walsh sense \cite{Walsh}. The following existence and uniqueness result can be regarded as our first main theorem.
\begin{theorem}\label{t4.2}
Suppose that the Fourier transform $\widehat\gamma = \mathcal{F} \gamma$ satisfies the fractional Dalang's condition:
\begin{equation}
\label{eq:dalang}
\int_{\RR^d} \frac{\widehat{\gamma}(d\xi)}{\beta+|\xi|^\alpha} < \infty,
\end{equation}
for some (and hence for all)  $\beta>0$.
Then  equation (\ref{1}) admits a unique mild solution given by \eqref{mild}.
\end{theorem}
\begin{remark}
We present our results only in the case of the initial condition $u(0, x)\equiv 1$ which makes our presentation and notation easier. We stress however, that with little bit extra effort our results could be extended to cover more general initial conditions. Indeed, our existence result can be generalised to cover even the cases of initial conditions given by measures (satisfying certain suitable conditions), by following the lines of \cite{CK}. Similarly, in the spirit of \cite[Corollary 3.3]{HNVZ19}, our approximation results can be generalised to the case of $u(0,x) = f(x)$ with suitable assumptions on $f$, once a comparison principle is established.
\end{remark}

Throughout the article, for a function $f$ and a (signed) measure $\mu$ we denote by $f\ast \mu$ the convolution defined by 
\begin{equation}
\label{eq:convolution}
(f \ast \mu)(y) = \int_{\RR^d} f(y-x)d\mu(x),
\end{equation}
provided it exists. If $\mu$ is absolutely continuous, then $d\mu(x) = \mu(x)dx$ for some function $\mu$ and we recover the classical convolution for integrable functions
$$
(f \ast \mu)(y) = \int_{\RR^d} f(y-x)\mu(x)dx.
$$
If $\mu$ can be viewed as a function, the well-known Young convolution inequality states that for $\frac{1}{p} + \frac{1}{q} = \frac{1}{r}+1$ with $1\leq p, q\leq r\leq \infty$, we have
\begin{equation}
\label{eq:young}
\Vert f\ast \mu \Vert_{L^{r}(\RR^d)} \leq \Vert f\Vert_{L^{p}(\RR^d)} \Vert \mu \Vert_{L^{q}(\RR^d)}.
\end{equation}
In particular, this gives us, for any $p\geq 1$, 
\begin{equation}
\label{eq:young-lp-l1}
\Vert f\ast \mu \Vert_{L^{p}(\RR^d)} \leq \Vert f\Vert_{L^{p}(\RR^d)} \Vert \mu \Vert_{L^{1}(\RR^d)}.
\end{equation}
More generally, if $\mu$ is a finite measure, a simple mollification argument shows that \eqref{eq:young-lp-l1} remains valid with $\Vert \mu \Vert_{L^{1}(\RR^d)}$ replaced by $\mu(\RR^d)$ (see, e.g. \cite[Proposition 3.9.9]{Bogachev}). Finally, by $I_{d-\beta}$ we denote the Riesz potential defined by, for $0<\beta < d$, 
$$
(I_{d-\beta}f)(x) = \int_{\RR^d} f(y)|x-y|^{-\beta}dy = (K_{d-\beta} \ast f)(x),
$$
where $K_{d-\beta}(y) = |y|^{-\beta}$. More generally, Riesz potential $I_{d-\beta} \mu$ with respect to a measure $\mu$ is defined through the convolution 
$$
\left(I_{d-\beta} \mu\right)(x) = (K_{d-\beta} \ast \mu)(x) = \int_{\RR^d} |x-y|^{-\beta}d\mu(y).
$$
In order to simplify our notation, we also define $I_{d-\beta}$ for $\beta=d$ simply as an identity operator. 

We also provide approximation results for the spatial average over an Euclidean ball of radius $R$, denoted by $B_R$. 
For these purposes we require some more refined information on the covariance $\gamma$ instead of the general condition \eqref{eq:dalang}.
\begin{assumption}
\label{assumption:main}
 We assume that $\gamma$ is given by 
the Riesz potential $\gamma = I_{d-\beta} \mu$, where $0<\beta\leq d$ and $\mu$ is a finite symmetric measure. 
Moreover, one of the following conditions holds:
\begin{itemize}
\item[(i)]  $\beta < \alpha \wedge d$.
\item[(ii)] $\beta =d=1$ and $\alpha >1$.
\item[(iii)] $\beta=d \ge \alpha$ and $\mu=\gamma$ is absolutely continuous, i.e. $d\gamma(x)= \gamma(x)dx$, with $\gamma \in L^r(\R^d)$ for some $r>\frac d\alpha$. In addition, if $r>2$, we impose Dalang's condition  \eqref{eq:dalang}.
\end{itemize}
\end{assumption}

\begin{remark}
Condition $\beta <\alpha $ in  Case (i) implies that Dalang's condition  \eqref{eq:dalang} is satisfied. Indeed,
 we recall that a Fourier transform $\widehat{\mu}$ of a finite measure $\mu$ is a bounded continuous function. 
Consequently, by recalling the convolution theorem $\widehat{f \ast \mu} = \widehat{f}\widehat{\mu}$ and the fact that the Riesz potential $I_{d-\beta}$ is a Fourier multiplier, we obtain 
\begin{equation}
\label{eq:fourier-gamma}
\widehat{\gamma}(d\xi) =c_{d,\beta} |\xi|^{\beta-d}\widehat\mu(\xi)d\xi,
\end{equation}
from which we deduce \eqref{eq:dalang}. Dalang's condition  \eqref{eq:dalang}  clearly holds in Case (ii). 
Finally, in Case (iii) we can deduce  \eqref{eq:dalang} from the Hausdorff-Young inequality if $r\le 2$.
\end{remark}
\begin{remark}
By carefully examining our proof one can see that our results remains valid provided that $\gamma = I_{d-\beta}\mu$ satisfies Dalang's condition and the statement in Proposition \ref{p1} holds for suitable number $2q$. 
\end{remark}
Case (ii) covers the case of the space-time white noise, where $\gamma$ is given by the Dirac delta $\gamma(y) = \delta_0(y)$. 
The case  $\gamma(y) = |y|^{-\beta}$  corresponds to the noise with spatial correlation given by the Riesz kernel, studied in the heat equation case $\alpha =2$ in \cite{HNVZ19}. In our terminology, this is included in Case (i) where   $\gamma = I_{d-\beta}\delta_0$.

Recall that the total variation distance between random variables (or associated probability distributions) is given by 
 \begin{equation}
 d_{\rm TV}(F, Z) :  = \sup\Big\{ P(F\in A) -  P(Z\in A) \,:\, A \subset \R  \,\,\,\, \text{Borel sets} \Big\}.  \label{TV_def}  
 \end{equation}
Our first main results concern  the following two quantitative central limit theorems for the spatial average of the solution. 
\begin{theorem}\label{t1.2}
Let $\gamma$ satisfy Assumption \ref{assumption:main} and let $u$ be the solution to the stochastic fractional heat equation (\ref{1}). Then for every $t>0$ there exists a constant $C$, depending solely on $t$, $\alpha$, $\sigma$, and the covariance $\gamma$, such that 
\[
d_{\rm TV}\left(   \frac{1}{\sigma_R} \int_ {B_R} \big[ u(t,x) - 1 \big]\,dx, ~Z\right) \leq CR^{-\frac{\beta}{2}}\,,
\]
where $Z\sim N(0,1)$ is  a standard normal random variable, and $\sigma_R^2 = {\rm Var} \big( \int_ {B_R} [ u(t,x) - 1 ]\,dx \big) \sim R^{2d-\beta}$,
as $R\to \infty$.
\end{theorem}
\begin{remark}
While we have stated our result concerning only a ball $B_R$ centered at the origin, we stress that with exactly the same arguments, one can replace $B_R$ with some other body $RA_0 = \{Ra : a \in A_0\}$. This affects only the normalization constants. Moreover, as in the heat case (cf. \cite[Remark 3]{HNVZ19}), one can allow the center of the ball $a_R$ to vary in $R$ as well. This fact follows easily from the stationarity. 
\end{remark}
Following the spirit of the mentioned references, we also provide functional version of Theorem \ref{t1.2}. 

\begin{theorem}  \label{t1.4}
Let $\gamma$ satisfy Assumption \ref{assumption:main} and let $u$ be the solution to the stochastic fractional heat equation (\ref{1}). Then
$$
\left\{  R^{\frac{\beta}{2}-d}\int_{B_R} \big[ u(t,x) -1 \big] \,dx   \right\}_{t\in [0,T]} \Rightarrow \left\{ \int_0^t  \varrho(s)  dY_s\right\}_{t\in [0,T]} \,,
$$
as $R\to \infty$, 
where $Y$ is a standard Brownian motion, the weak convergence takes place on the space of continuous functions $C([0,T])$, and $\varrho(s)$ is given by;
\begin{itemize}
\item If $\beta < d$, then $\varrho(s) = \sqrt{\mu\left(\RR^d\right)\int_{B_1^2}|x-x'|^{-\beta}dxdx'}\E[\sigma(u(s,y))]$.
\item If $\beta =d$, then $\varrho(s) = \sqrt{|B_1|\int_{\RR^d} \E\left[\sigma(u(s,0))\sigma(u(s,z))\right]d\mu(z)}$.
\end{itemize}
\end{theorem}
\begin{remark}
We prove later (see Lemma \ref{lemma:nu-okay}) that
$$
\int_{\RR^d} \E\left[\sigma(u(s,0))\sigma(u(s,z))\right]d\mu(z) \geq \left[\E[\sigma(u(s,y))]\right]^2.
$$
Under our initial condition $u(0,x) \equiv 1$, we may hence apply the arguments of \cite[Lemma 3.4]{DNZ18} to see the equivalence
$$
\sigma(1)=0 \Leftrightarrow \sigma_R = 0,\forall R>0 \Leftrightarrow \sigma_R=0 \mbox{ for some }R>0 \Leftrightarrow \lim_{R\to \infty} \sigma_R^2 R^{\beta-2d} = 0.
$$
Hence $\sigma(1)\neq 0$ is a natural condition that guarantees $\sigma_R>0$ for all $R>0$. 
Note also that $\sigma(1)\neq 0$ is necessary to exclude the trivial solution $u(t,x)\equiv 1$ by using the Picard iteration.
\end{remark}
\begin{example}
\item Suppose $\mu = \delta_0$ and let $\beta= d = 1$ and $\alpha>1$. This case corresponds to the space-time white noise, and now 
$$
\varrho(s) = \sqrt{|B_1|\int_{\RR^d}\E\left[\sigma(u(s,0))\sigma(u(s,z))\right]d\mu(z)} = \sqrt{2\E\sigma^2(u(s,0))}.
$$
In the case $\alpha=2$, we thus recover the results of \cite{HNV18}.
\end{example}
\begin{example}
Suppose $\beta<d$ and let $\mu = \delta_0$. This case corresponds to the white-colored case with the spatial covariance given by the Riesz kernel. Now
$$
\varrho(s) = \sqrt{\int_{B_1^2}|x-x'|^{-\beta}dxdx'}\E[\sigma(u(s,y))]
 $$
 and consequently, for $\alpha = 2$ we obtain the results of \cite{HNVZ19}.
\end{example}
\begin{remark}
We emphasis that the additional parameters associated to the fractional operator (i.e. $\alpha$) does not affect the above results, except for the constant quantities through the solution $u$. Indeed, the renormalization rate and the total variation distance are, up to multiplicative constants, the same as in the case $\alpha=2$ corresponding to the classical stochastic heat equation. 
\end{remark}

\section{Preliminaries}\label{sec:prel}
In this section we present some preliminaries that are required for the proofs of our main theorems. In particular, we recall some facts on Malliavin calculus and Stein's method together with some basic properties of the fractional Green kernel. Finally, in Proposition \ref{p1} we present a basic inequality that allows us to derive a bound for the Malliavin derivative.
\subsection{Malliavin calculus and Stein's method} We  start by introducing the Gaussian noise that governs the stochastic  fractional heat equation  (\ref{1}). 

Denote by $C_{c}^{\infty}\left( [0, \infty) \times \mathbb{R}^d\right)$ the class of $C^{\infty}$ functions on $ [0, \infty) \times \mathbb{R}^d$ with compact support.  We consider a Gaussian family of centered random variables  
\[
\left( W(\varphi), \varphi \in C^{\infty}_{c} \left( [0, \infty) \times \mathbb{R}^d\right)\right)
\]
on some complete probability space $\left(\Omega, \mathcal{F}, P\right) $ such that 
\begin{equation}
\label{cov-col}
\E [W(\varphi) W (\psi) ]= \int_{0} ^{\infty} \int_{\mathbb{R}^d} \int_{\mathbb{R}^d}\varphi(s, y) \psi (s, y') \gamma(y-y')dy dy'ds := \langle \varphi, \psi \rangle _{\HH}.
\end{equation}
We stress again that, in general, $\gamma$ is not a function, and hence \eqref{cov-col} should be understood as 
\begin{equation}
\label{cov-col-general}
\E [W(\varphi) W (\psi) ]= \int_{0} ^{\infty} \int_{\mathbb{R}^d} \varphi(s, y) \left[\psi(s,\cdot) \ast \gamma\right](y) dy ds.
\end{equation}

We denote by $\HH $ the Hilbert space defined as the closure of $C_{c}^{\infty}\left( [0, \infty) \times \mathbb{R}^d\right)$ with respect to the inner product (\ref{cov-col}). By density, we obtain an isonormal process $(W(\varphi), \varphi \in \HH )$, which consists of  a Gaussian family of centered random variable such that, for every $\varphi, \psi \in \HH$,
\begin{equation*}
\E[W(\varphi) W (\Psi) ]= \langle \varphi, \psi \rangle _{\HH}.
\end{equation*}
The Gaussian family $(W(\varphi), \varphi \in \HH )$ is usually called a white-colored noise because it behaves as a Wiener process with respect to the time variable $t\in [0, \infty)$ and it has a spatial covariance given by   the measure $\gamma$. 

Let us introduce the filtration associated to the random noise $W$. For $t>0$, we denote by $\mathcal{F}_{t}$ the sigma-algebra  generated by the random variables $W(\varphi)$, with $\varphi\in \HH$ having its support   included in the set $[0, t]\times \mathbb{R}^d$. For every random field $(X(s,y), s\geq 0, y\in \mathbb{R}^d)$, jointly measurable  and adapted  with respect to the filtration $\left( \mathcal{F}_{t}\right) _{t\geq 0} $,  satisfying 
\begin{equation*}
\E\left[ \Vert X\Vert _{\HH} ^{2}\right] <\infty,
\end{equation*}
we can define stochastic integrals with respect to $W$ of the form 
\begin{equation*}
\int_{0} ^{\infty} \int_{\mathbb{R}^d} X(s, y) W (ds, dy)
\end{equation*}
in the sense of Dalang-Walsh (see \cite{Da} and \cite{Walsh}). This integral satisfies the It\^o-type isometry
\begin{equation}\label{inte}
\E \left[\left( \int_{0} ^{\infty} \int_{\mathbb{R}^d} X(s, y) W (ds, dy)\right) ^{2} \right]= \E\left[ \Vert X\Vert _{\HH} ^{2}\right]. 
\end{equation}
The Dalang-Walsh integral also satisfies the following version of the Burkholder-Davis-Gundy inequality: for any $t\geq 0$ and $p\geq 2$,
\begin{eqnarray}
&&\left| \left|  \int_{0} ^{\infty} \int_{\mathbb{R}^d} X(s, y) W (ds, dy)\right|\right|  _{p} ^{2}\nonumber \\
&&\leq c_{p} \int_{0} ^{t} \int_{\mathbb{R}^d}\int_{\mathbb{R}^d }\Vert X(s, y) X(s, y')\Vert  _{\frac{p}{2}} \gamma(y-y')dydy'ds. \label{burk2}
\end{eqnarray}

Leu us next describe the basic tools from Malliavin calculus needed in this work. We introduce
$C_p^{\infty}(\RR^n)$ as the space of smooth functions with all their partial derivatives having at most polynomial growth at infinity, and $\mathcal{S}$ as the space of simple random variables of the form 
\begin{equation*}
F = f(W(h_1), \dots, W(h_n)),
\end{equation*}
where $f\in C_p^{\infty}(\RR^n)$ and $h_i \in \HH$, $1\leq i \leq n$. Then the Malliavin derivative $DF$ is defined as $\HH$-valued random variable
\begin{equation}
DF=\sum_{i=1}^n  \frac {\partial f} {\partial x_i} (W(h_1), \dots, W(h_n)) h_i\,.
\end{equation}
For any $p\geq 1$, the operator $D$  is closable as an operator  from $L^p(\Omega)$ into $L^p(\Omega;  \HH)$. Then $\mathbb{D}^{1,p}$ is defined as the completion of $\mathcal{S}$ with respect to the norm
\begin{equation*}
\|F\|_{1,p} = \left(\E [|F|^p] +   \E(  \|D F\|^p_\HH )   \right)^{1/p}\,.
\end{equation*}
The adjoint operator $\delta$ of the derivative is defined through the duality formula
\begin{equation}\label{eq: duality formula}
\E (\delta(u) F) = \E( \langle u, DF \rangle_\HH),
\end{equation}
valid for any $F \in \mathbb{D}^{1,2}$ and any $u\in {\rm Dom} \, \delta \subset L^2(\Omega; \HH) $. The operator $\delta$ is
also called the Skorokhod integral since, in the case of the standard Brownian motion,
it coincides
with an extension of the It\^o integral introduced by Skorokhod (see, e.g. \cite{GT, NuPa}). 
In our context, any   adapted random field $X$ which is jointly measurable and satisfies (\ref{inte}) belongs to the domain of $\delta$, and $\delta (X)$ coincides with the Walsh integral:
\[
\delta (X) = 
\int_0^\infty \int_{\R} X(s,y) W(d s, d y).
\]
This allows us to represent the solution $u(t,x)$ to \eqref{1} as a Skorokhod integral. 

\medskip

The proofs of our main results are based on  
 Malliavin-Stein approach, introduced by Nourdin and Peccati in \cite{NP09} (see also the book \cite{NP}). In particular, we apply the following result to obtain rate of convergence in the total variation distance (see \cite{Zhou} and also \cite{HNV18,Eulalia}). 

\begin{proposition}  \label{lem:dist}
If $F$ is a centered  random  variable in the Sobolev space $\mathbb{D}^{1,2}$ with unit variance such that  $F = \delta(v)$ for  some $\HH$-valued random variable $v$ belonging to the domain of $\delta$,     then,  with $Z \sim N (0,1)$,
\begin{equation}\label{eq:dist}
d_{\rm TV}(F, Z) \leq 2 \sqrt{{\rm Var} \left(  \langle DF, v\rangle_\HH  \right) }.
\end{equation}
\end{proposition}

\subsection{On fractional Green kernel}
We recall some useful properties of the kernel $G_{\alpha}$ defined through \eqref{FG}. For details, we refer to \cite{DD,CD,CHN}.

\begin{enumerate}

\item For every $t>0$,  $G_{\alpha}(t, \cdot) $ is the density of a $d$-dimensional  L\'evy stable process at time $t$. In particular, we have
\begin{equation}\label{density}
\int_{\mathbb{R}^d} G_{\alpha}(t,x)dx=1.
\end{equation}

\item For every $t$, the kernel $G_{\alpha}(t,x)$ is real valued, positive, and symmetric in $x$.

\item The operator $G_{\alpha} $ satisfies the semigroup property, i.e. 

\begin{equation}
\label{semig}
G_{\alpha} (t+s, x)= \int_{\mathbb{R}^{d}} G_{\alpha} (t, z ) G_{\alpha} (s, x-z) dz
\end{equation}
for $0<s<t$ and $x\in \mathbb{R}^d$.

\item $G_{\alpha}$ is infinitely differentiable with respect to $x$, with all the derivatives bounded and converging to zero as $\vert x\vert \to \infty$. Moreover, we have the scaling property 
\begin{equation}
\label{scale}
G_{\alpha}(t,x)= t^ {-\frac{d}{\alpha} } G_{\alpha}(1, t ^ {-\frac{1}{\alpha}}x).
\end{equation}

\item There exist two constants $0<K_{\alpha}' < K_{\alpha} $  such that 
\begin{equation}
\label{29m-1}
K' _{\alpha} \frac{1}{\left(1+ \vert x\vert \right)^ {d+\alpha}}\leq \left| G_{\alpha} (1,x)\right| \leq  K_{\alpha} \frac{1}{ \left(1+ \vert x\vert\right) ^ {d+\alpha}}
\end{equation}
for all $x\in \mathbb{R}^d$. Together with the scaling property, this further translates into
\begin{equation}  \label{Eq1}
K' _{\alpha} \frac{t^{-\frac{d}{\alpha}}}{\left(1+ \vert t^{-\frac{1}{\alpha}}x\vert \right)^ {d+\alpha}}\leq \left| G_{\alpha} (t,x)\right| \leq  K_{\alpha} \frac{t^{-\frac{d}{\alpha}}}{ \left(1+ \vert t^{-\frac{1}{\alpha}}x\vert\right) ^ {d+\alpha}}.
\end{equation}
\end{enumerate}

\subsection{A basic inequality}
The following proposition contains an inequality that plays a fundamental role in the proof of the estimates of the $p$-norm of the Malliavin derivative.

\begin{proposition} \label{p1}
Suppose that the covariance $\gamma$ satisfies Assumption \ref{assumption:main}.
Then,  there exists a number $2q \in \left(1,\frac{2d}{2d-\alpha} \wedge \frac{d+\alpha}{d}\right)$ such that for any functions $f,g \in L^{2q}(\mathbb{R}^d)$ we have 
\begin{equation}
\label{eq:key-convolution-bound}
\int_{\mathbb{R}^d} f(y) \left[g \ast \gamma\right](y)dy \leq C \Vert f\Vert_{L^{2q}(\mathbb{R}^d)}\Vert g\Vert_{L^{2q}(\mathbb{R}^d)}.
\end{equation}
 \end{proposition}

\begin{remark}
The requirement 
$2q < \frac{2d}{2d-\alpha}$ ensures that 
\begin{equation} \label{kappa}
\kappa =\frac{2d}{\alpha}\left(1-\frac{1}{2q}\right)<1,
\end{equation}
 while the requirement $2q < \frac{d+\alpha}{d}$ ensures that $G^{\frac{1}{2q}}(1,x)$ is integrable. Note also that $\frac{d+\alpha}{d} \leq \frac{2d}{2d-\alpha}$ only if $d\leq \alpha$. Since $\alpha\leq 2$, this can happen only in the one-dimensional case $d=1$ or in the heat case $\alpha=2$ and $d=1,2$. In the latter however, $G^{\frac{1}{2q}}(1,x)$ is integrable regardless of the value $2q$ and consequently, our results can be applied in that case as well under a condition $2q \in \left(1,\frac{2d}{2d-2}\right)$. 
\end{remark}
 
 \begin{proof}[Proof of Proposition \ref{p1}]
 We decompose the proof into the three possible cases from Assumption \ref{assumption:main}:
 
 \medskip
 \noindent
 {\it Case (i)}: Taking $2q=\frac {2d} {2d-2\beta}$ (recall $\beta < \alpha\wedge d$) and using H\"older's inequality, we obtain
 \[
 \int_{\R^{2d}} f(x) [ g*\gamma] (x) dx \le  \| f\|_{L^{2q}(\R^d)} \| g*\gamma\|_{L^{2q/(2q-1)}(\R^d)}.
 \]
 Notice that  $g*\gamma=( I_{d-\beta} g) * \mu$. Therefore, it follows from \eqref{eq:young-lp-l1} that 
 \[
 \| g*\gamma\|_{L^{2q/(2q-1)}(\R^d)} \le  \mu(\R^d)  \| I_{d-\beta} g \|_{L^{2q/(2q-1)}(\R^d)}.
 \]
 We then conclude the proof using the fact that $2q=\frac {2d} {2d-2\beta}$ and applying the following  Hardy-Littlewood-Sobolev inequality (see e.g. \cite{Lieb} and references therein):  for $1<p<r<\infty$ satisfying $\frac{1}{r} = \frac{1}{p} - \frac{d-\beta}{d}$, we have
\begin{equation}
\label{eq:HLS-lp-mapping}
\| I_{d-\beta} g\|_{L^{r}(\R^d)} \leq C\| g\|_{L^{p}(\R^d)}.
\end{equation}

 \medskip
 \noindent
 {\it Case (ii)}:   Suppose $\beta=d$.
  By Young's inequality \eqref{eq:young-lp-l1} and H\"older's inequality, we get 
$$
\int_{\mathbb{R}^d} f(x) \left[g \ast \gamma\right](y)dy \leq \Vert f\Vert_{L^{2}(\mathbb{R}^d)}\Vert g \ast \gamma \Vert_{L^{2}(\mathbb{R}^d)} \leq C \Vert f\Vert_{L^{2}(\mathbb{R}^d)}\Vert g \Vert_{L^{2}(\mathbb{R}^d)}.
$$
Consequently, one can always choose $q=1$ in \eqref{eq:key-convolution-bound}. However, then $2q < \frac{2d}{2d-\alpha}\wedge \frac{d+\alpha}{d}$ only if $\alpha > d$. Taking into account the fact $\alpha \in (0,2]$, this forces $d=1$ and $\alpha>1$. In conclusion, in the one-dimensional case and for $\alpha>1$ we obtain the estimate  (\ref{eq:key-convolution-bound}) with $q=1$, which
 completes the proof of Case (ii).

 \medskip
 \noindent
 {\it Case (iii)}:  
Let $\beta=d\ge \alpha$ and suppose that $\gamma$ is absolutely continuous with a density $\gamma \in L^r(\R^d)$, where $r>\frac d\alpha$.
In this case we choose $2q=\frac {2r}{2r-1}$. Clearly $2q>1$. Moreover, condition $r>\frac d\alpha$ implies $2q < \frac {2d}{2d-\alpha}$ and
$\frac {2d}{2d-\alpha} \le \frac {d+\alpha} d$ because $d \ge \alpha$. Finally, 
 H\"older's inequality and Young's inequality \eqref{eq:young} gives us
\[
\int_{\R^d} f(x) g(y)  \gamma(x-y) dx dy \le \| f\|_{L^{2q}(\R^d)} \| g\|_{L^{2q}(\R^d)}\| \gamma\|_{L^{r}(\R^d)}.
\]
\end{proof}

\section{Proof of Theorem \ref{t4.2}}
\label{sec:existence-proof}
For $t\geq 0$, we denote
\begin{equation}
\label{itx}
I(t)= \int_{\mathbb{R}^{d} } \int_{\mathbb{R}^{d} }  G_{\alpha} (t,y )G_{\alpha} (t,y' )\gamma( y-y ')dy'dy.
\end{equation}
Taking the Fourier transform and using (\ref{FG}), we see that $I(t)$ can equally be given by 
\begin{equation}
\label{itx:fourier}
I(t)= \int_{\mathbb{R}^{d}} e ^{-2t\vert \xi \vert ^{\alpha}} \widehat{\gamma}(d\xi).
\end{equation}
Suppose now that $\gamma$ satisfies \eqref{eq:dalang}. For $\beta>0$, we define a function $\Upsilon(\beta)$ by
$$
\Upsilon(\beta) := \int_0^\infty e^{-\beta t}I(t)dt = \int_{\RR^d} \frac{\widehat{\gamma}(d\xi)}{\beta+2 |\xi|^\alpha}.
$$
Clearly, $\Upsilon$ is non-negative, decreasing in $\beta$, and $\lim_{\beta\to \infty} \Upsilon(\beta)=0$. 

Before proving Theorem \ref{t4.2} we introduce the following technical lemma that can be viewed as a fractional version of \cite[Lemma 2.5]{CK}.
\begin{lemma}
\label{lemma:convergence}
Let $I(t)$ be given by \eqref{itx} and, for given $\iota>0$, let $h_n$ be defined recursively by $h_0(t) = 1$, and for $n\geq 1$
$$
h_{n}(t) = \iota\int_0^t h_{n-1}(s)I(t-s)ds.
$$
Then for any $p\geq 1$ and any fixed $T<\infty$, the series
\begin{equation}
\label{eq:lp-convergence}
H(\iota,p,t) :=\sum_{n\geq 0} \left[h_n(t)\right]^{\frac{1}{p}}
\end{equation}
converges uniformly in $t\in[0,T]$.
\end{lemma}
\begin{proof}
By the same argument as in the proof of \cite[Lemma 2.5]{CK}, we get, for any $\beta>0$, that
$$
\int_0^\infty e^{-\beta t}h_n(t) dt = \frac{1}{\beta}\left(\iota \int_0^\infty e^{-\beta t} I(t)dt\right)^n = \frac{1}{\beta}\left[\iota\Upsilon(\beta)\right]^n.
$$
By choosing $\beta$ large enough, we have $\iota\Upsilon(\beta)\le 1/2$ and, as in \cite{CK}, by choosing the  smallest such $\beta$ this gives us $H(\iota,1,t) \leq \exp(Ct)$ for some constant $C$ depending on $\Upsilon$ and $\iota$. Similarly, for the general case $p>1$ we may apply H\"older inequality to get 
$$
\int_0^\infty e^{-\beta t}h^{1/p}_n(t) dt \leq \beta^{-\frac{1}{q}} \left(\int_0^\infty e^{-\beta t}h_n(t) dt\right)^{\frac{1}{p}} = \frac{1}{\beta}\left[\iota\Upsilon(\beta)\right]^{\frac{n}{p}},
$$
where $\frac 1p + \frac 1q=1$.
Hence, similar arguments show that $H(\iota,p,t)\leq \exp(Ct)$ and, in particular, that the series in \eqref{eq:lp-convergence} converges.
\end{proof}
Equipped with Lemma \ref{lemma:convergence},  we are now able to prove Theorem \ref{t4.2}.
\begin{proof}[Proof of Theorem \ref{t4.2}] 
Define the standard Picard iterations by setting $u_{0}(t,x)=1$ and, for $n\geq 1$,
\begin{equation*}
u_{n+1} (t,x)=u_{0}(t,x)+\int_{0} ^ {t}\int_{\mathbb{R}^{d}} G_{\alpha} (t-s,x -y ) \sigma (u_{n}(s, y))W (ds, dy), \hskip0.5cm t\geq 0, x\in \mathbb{R}^{d}.
\end{equation*}  
By induction, we can easily show that for every $n\geq 0$,  $u_{n}(t, x)$ is well-defined and, for every $p\geq 2$ and $\beta>0$, we have
\begin{equation}\label{22i-1}
\sup_{t\in [0, T]} \sup_{x\in \mathbb{R}^{d}} \E\left[e^{-p\beta t}\left| u_{n}(t,x) \right| ^p\right]<\infty.
\end{equation}
This in turn shows that
\begin{equation}\label{22i-needed}
\sup_{t\in [0, T]} \sup_{x\in \mathbb{R}^{d}} \E\left[ \left| u_{n}(t,x) \right| ^ {p} \right]<\infty.
\end{equation}
To see (\ref{22i-1}), we first observe that it is clearly true for $n=0$. Suppose now that it holds for some $n$. We have
$$
e^{-\beta t} u_{n+1} (t,x)=e^{-\beta t} u_{0}(t,x)+\int_{0} ^ {t}e^{-\beta t}\int_{\mathbb{R}^{d}} G_{\alpha} (t-s,x -y ) \sigma (u_{n}(s, y))W (ds, dy)
$$
and, for every $p\geq 2$,  by using \eqref{inte} and (\ref{burk2}), 
\begin{eqnarray*}
\E \left[e^{-p\beta t}\left|u_{n+1}(t,x)\right|^p \right] &\leq &  C\left( 1+  \Big\| \int_{0} ^{t} e^{-\beta t}\int_{\mathbb{R}^{d}} \int_{\mathbb{R}^{d}}G_{\alpha} (t-s,x -y )G_{\alpha} (t-s,x -y' ) \right.\\
&&\phantom{kukkuu}\left.\times\sigma(u_{n}(s, y)) \sigma (u_{n}(s, y')) \gamma(y-y ')dy'dy\Big\| _{\frac{p}{2}}^{\frac{p}{2}}\right)\\
&\leq &  C \left[ 1+  \left( \int_{0} ^{t} e^{-\beta (t-s)}\int_{\mathbb{R}^{d}} \int_{\mathbb{R}^{d}}G_{\alpha} (t-s,x -y )G_{\alpha} (t-s,x -y' )\right.\right.\\
&&\phantom{kukkuu}\left.\left.\times
e^{-\beta s}\left| \left| \sigma(u_{n}(s, y)) \sigma (u_{n}(s, y'))\right| \right| _{\frac{p}{2}}\gamma(y-y ')dy'dy\right)\right] ^{\frac{p}{2}}.
\end{eqnarray*}
By using the Lipschitz assumption on $\sigma$ and the induction hypothesis we get
$$
\E \left[e^{-p\beta s}\vert  \sigma (u_{n}(s,y)) \vert ^p\right] \leq    C\left(1+ \sup_{s\in [0,T],y\in \mathbb{R}^d} \E \left[e^{-p\beta s}\vert u_{n}(s, y)\vert ^p  \right]  \right) <\infty.
$$ 
Hence 
$$
\sup_{s\in [0,T]}\sup_{y,y'\in \RR^d} e^{-\beta s}\left| \left| \sigma(u_{n}(s, y)) \sigma (u_{n}(s, y'))\right| \right| _{\frac{p}{2}} < \infty
$$
and we obtain 
\begin{align*}
\E \left[e^{-p\beta t}\left|u_{n+1}(t,x)\right| ^p\right]  &\leq C \left( 1+ \int_{0}^{t} e^{-\beta(t-s)}I (t-s) ds \right)^{\frac{p}{2}} \\
& \leq C \left(1+ \int_0^T e^{-\beta s}I(s)ds\right)^{\frac{p}{2}}  < \infty.
\end{align*}
Applying similar arguments together with H\"older's inequality for 
$$
H_{n}(t):= \sup_{x \in \mathbb{R}^{d}} \E \left[\left| u_{n+1} (t, x)- u_{n}(t,x) \right| ^ {p}\right]
$$
gives us
\begin{align*}
H_n(t)&\leq  C \left[  \int_{0} ^ {t} \int_{\mathbb{R}^{d}  }  \int_{\mathbb{R}^{d}  }  G_{\alpha} (t-s,x -y )G_{\alpha} (t-s,x -y' )\left| \left| \sigma (u_{n}(s, y))-  \sigma (u_{n-1}(s, y))\right|\right| _{p}\right.\\
&\phantom{kukkuu}\times\left.
  \left|  \left| \sigma (u_{n}(s, y'))-  \sigma (u_{n-1}(s, y'))\right|\right|_{p} \gamma(y-y')dy'dyds\right] ^ {\frac{p}{2}}\\
&\leq  C \int_{0} ^ {t} \int_{\mathbb{R}^{d}  }   \int_{\mathbb{R}^{d}  }  G_{\alpha} (t-s,x -y )G_{\alpha} (t-s,x -y' ) H_{n-1}(s) \gamma(y-y')dy'dyds\\
&\leq  C \int_{0} ^{t} I(t-s) H_{n-1}(s)ds.
\end{align*}
By standard arguments, it suffices to consider the case of an equality. In this case, it follows from Lemma \ref{lemma:convergence} that $\sum_{n\geq 1} H_{n} (t)^ {\frac{1}{p}}$ converges uniformly on $[0, T]$. Consequently, the sequence $u_{n}$ converges in $L^ {p}(\Omega) $, uniformly on $[0, T] \times \mathbb{R}^{d} $, and its limit satisfies (\ref{mild}).
The uniqueness follows in a similar way, and the stationarity of the solution with respect to the space variable is a consequence of the proof of Lemma 18 in \cite{Da}.

\end{proof}

\section{Proofs of Theorem \ref{t1.2} and Theorem \ref{t1.4}}
\label{sec:clt-proofs}
In this section we prove Theorems \ref{t1.2} and \ref{t1.4}. The key ingredient for the proofs is to bound the Malliavin derivative of the solution to (\ref{1}) by a quantity involving the Green kernel associated to the fractional operator (\ref{FG}). Once a suitable bound is established, it suffices to study the asymptotic variance and follow the ideas presented in \cite{HNV18,HNVZ19}. We divide this section into four subsections. In the first one we study the (bound for the) Malliavin derivative of the solution, and in the second we study the correct normalization rate. The last two subsections are devoted to the proofs of Theorem \ref{t1.2} and Theorem \ref{t1.4}.

\subsection{Bound for the Malliavin derivative}

We begin by providing a linear equation for the Malliavin derivative of the solution. The claim follows from \eqref{mild}, and the proof is rather standard. For this reason we omit the details. 
\begin{proposition}\label{pp11}
Let $u$ be the mild solution to (\ref{1}). Then for every  $t\in (0, T]$, $p\ge 2$  and $x \in \mathbb{R}^d $, the random variable  $u(t, x)$ belongs to the  Sobolev space
$\mathbb{D}^{1,p}$  and its Malliavin derivative satisfies
\begin{eqnarray}
&&D_{r, z }u(t, x)= G_{\alpha} (t-r, x-z) \sigma (u(r, z))\nonumber\\
&&+ \int_{r} ^ {t}\int_{\mathbb{R}   } G_{\alpha} (t-s, x-y) \Sigma (s, y) D_{r, z } u(s, y) W(ds, dy),\label{22i-2}
\end{eqnarray}
where $\Sigma (r, z) $ is an adapted and bounded (uniformly with respect to $r$ and $z$) stochastic process that coincides with $\sigma ' (u(r, z)) $ whenever $\sigma $ is differentiable.
\end{proposition}

The following result provides a bound for the $p$-norm of the Malliavin derivative of the solution.

\begin{proposition}
\label{prop:malliavin-bound}
Suppose that $\gamma$ satisfies Assumption \ref{assumption:main} and  recall (see (\ref{kappa})) that
$\kappa = \frac{2d}{\alpha}\left(1-\frac{1}{2q}\right)$, where $q$ is from Proposition \ref{p1}.
Then for every $0<s<t<T$, for every $x, y \in \mathbb{R}^d $, and for every $p\geq 2$ we have 
\begin{equation*}
\Vert D_{s,y}u(t,x)\Vert _{p}  \leq  c (t-s) ^{-\frac{\kappa}{2}} G^{\frac 1{2q}} _{\alpha} (t-s, x-y).
\end{equation*}
\end{proposition}
Proposition \ref{prop:malliavin-bound} is based on the following lemma which proof is postponed to the appendix.
\begin{lemma}\label{ll2}
Suppose that $\gamma$ satisfies Assumption \ref{assumption:main} and assume that $g: [0, T]\times \mathbb{R}^d\to  \mathbb{R}^d$ is non-negative function satisfying, for every $t\in [0, T] $ and ${x} \in \mathbb{R}^d $,
\begin{eqnarray}
g(t,{x} ) ^ {2} &\leq & G_{\alpha } (t,{x}) ^ {2} + \int_{0} ^ {t} \int_{\mathbb{R}^d  }  \int_{\mathbb{R}^d }    G_{\alpha }(t-s, {x}-{y} ) G_{\alpha }(t-s, {x}-{y}' ) \nonumber\\
&& \phantom{kukkuu}\times g(s, {y}) g(s, {y}' )  \gamma(y-y')dy'dyds \label{20i-1}.
\end{eqnarray}
Then
\begin{equation}\label{20i-2}
g(t, {x})  \leq c t ^{-\frac{\kappa}{2}}G _\alpha^{\frac 1{2q}} (t, {x}),
\end{equation}
where $\kappa = \frac{2d}{\alpha}\left(1-\frac{1}{2q}\right)$ and $q$ is from Proposition \ref{p1}.
\end{lemma}
\begin{proof}[Proof of Proposition \ref{prop:malliavin-bound}] 
In a standard way we can show that, for every $t\in (0, T]$ and $x \in \mathbb{R} $, the random variable  $u(t, x)$ belongs to the Sobolev space $\mathbb{D}^{1,p}$ for all $p\ge 2$ 
 and its Malliavin derivative satisfies (\ref{22i-2}). Moreover, using the Burkholder-Davis-Gundy inequality (\ref{burk2}) we obtain that, for any $p\geq 2$,
\begin{eqnarray*}
\Vert D_{r, z}u(t,x)  \Vert _{p} ^ {2} &\leq & C_{p}  G_{{\alpha}}  (t-r, x-z) ^ {2}\\
& + &  C_{p} \int_{r}^ {t} \int_{\mathbb{R}} \int_{\mathbb{R} }  G_{\alpha} (t-s, x-y) G_{\alpha} (t-s, x- y') \\
&& \times\Vert D_{r, z} u(s, y) \Vert _{p}  \Vert D_{r, z} u(s, y) \Vert _{p}\gamma (y-y ')dy'dyds.
\end{eqnarray*}
To conclude the proof, it suffices to apply Lemma \ref{ll2} with $\theta =t-r, \eta= x-z $, and 
$$g(\theta, \eta) = \Vert D_{r, z} u(\theta+r, \eta+  z)\Vert _{p}.$$
\end{proof}

For later use we also record the following simple technical fact. 
\begin{lemma}
\label{lma:2q-integral}
Suppose $2q \in \left(1,\frac{2d}{2d-\alpha} \wedge \frac{d+\alpha}{d}\right)$. Then
\[
\int_{\R^d} G^{\frac 1{2q}}_{ \alpha} (r-s, \eta) d\eta = C (r-s) ^{\frac{\kappa}{2}},
\]
where $\kappa$ is defined in  (\ref{kappa}).
\end{lemma}
\begin{proof}
By the scaling property \eqref{scale} we get 
\[
\int_{\R^d} G^{\frac 1{2q}}_{ \alpha} (r-s, \eta) d\eta = (r-s) ^{\frac{\kappa}{2}}\int_{\R^d} G^{\frac 1{2q}}_{ \alpha} (1, \eta) d\eta
\]
where, by \eqref{29m-1},
$$
\int_{\R^d} G^{\frac 1{2q}}_{ \alpha} (1, \eta) d\eta \leq C \int_{\R^d} \left(1+|\eta|\right)^{-\frac{d+\alpha}{2q}}d\eta < \infty
$$
since $\frac{d+\alpha}{2q}>d$.
\end{proof}

\subsection{Asymptotic behavior of the covariance}
\label{subsec:renormalization}

Let us use the following notation. For fixed $t>0,$ we define
\begin{equation}\label{not1}
G_{R}(t):=\int_{B_R}\left[ u(t,x)-1\right]dx  \mbox{ and } \varphi_{R} (t, y):=\int_{B_R} G_{\alpha} (t, x-y) dx .
\end{equation}
The constant $k_\beta$, for $\beta \le d$, is defined by
\begin{equation}\label{not2}
k_d = |B_1| \mbox{ and } k_{\beta}:=\int_{B_1^2} |x-x'|^{-\beta} dx dx',\quad \beta<d.
\end{equation}
Set
\begin{equation}\label{psi}
\Psi(s,  z) = \E [\sigma (u(s,  0)) \sigma (u(s,  z))]
\end{equation}
and
\begin{equation}\label{theta}
\theta_{\alpha} (s)=\E[\sigma(u(s,y))].
\end{equation}
When $\beta=d$, we put
\begin{equation}\label{psi-finite-meas}
\nu_{\alpha}(s) = \left(\int_{\RR^d} \Psi(s,z)d\mu(z)\right)^{\frac{1}{2}}
\end{equation}
and following lemma justifies the fact that $\nu_\alpha$ is well-defined. The proof is postponed to the end of this subsection. 
\begin{lemma}
\label{lemma:nu-okay}
Suppose that $\gamma$ satisfies Assumption \ref{assumption:main} with $\beta=d$ and let $\Psi$ be given by \eqref{psi}. Then for every $s\geq 0$ we have
$$
\int_{\RR^d} \Psi(s,z)d\mu(z) \geq 0.
$$
In particular, $\nu_\alpha$ given by \eqref{psi-finite-meas} is well-defined. Moreover, for every $s\in [0,T]$ we have
$$
\nu^2_\alpha(s) \geq \theta_\alpha^2(s).
$$
\end{lemma}
The following theorem provides us the correct renormalization as well as the limiting covariance.
\begin{theorem}\label{pp2}
Suppose that $\gamma$ satisfies Assumption \ref{assumption:main}. Then
$$
\lim_{R\to\infty }R^{\beta-2d}\E(G_R(t)  G_{R}(r)) = k_\beta \int_{0}^{t\wedge r}  \left[ \mu(\RR^d)\theta^2_{\alpha}(s)\mathbf{1}_{\beta<d} + \nu^2_{\alpha}(s)\mathbf{1}_{\beta=d} \right]ds.
$$
\end{theorem}

Before we proceed to the proof of Theorem \ref{pp2}, we present a couple of technical lemmas.
\begin{lemma}
\label{pro:correct-limit}
Suppose that $\gamma$ satisfies Assumption \ref{assumption:main}. Then for any bounded function $s\mapsto \theta(s)$ we have, as $R\to \infty$,
$$
R^{\beta-2d}\int_0^t \theta(s) \int_{\RR^{2d}} \varphi_R(t-s,y)\varphi_R(t-s,y')\gamma(y-y')dy'dyds \to k_{\beta}\mu(\RR^d)\int_0^t \theta(s)ds,
$$
where $k_\beta$ is defined in \eqref{not2}.
\end{lemma}
\begin{proof}
Recall that, writing formally by \eqref{cov-col-general}, we have
\begin{eqnarray*}
&&\int_{\RR^{2d}} \varphi_R(t-s,y)\varphi_R(t-s,y')\gamma(y-y')dy'dy \\
&=& \int_{\RR^{d}} \varphi_R(t-s,y)\left[\varphi_R(t-s,\bullet) \ast I_{d-\beta} \ast \mu\right](y)dy.
\end{eqnarray*}
Since clearly $\varphi_R(t-s,\bullet) \in L^1(\mathbb{R}^d) \cap L^{\infty}(\mathbb{R}^d)$, it follows from Young's inequality \eqref{eq:young-lp-l1} and Hardy-Littlewood-Sobolev's inequality \eqref{eq:HLS-lp-mapping} that $\varphi_R(t-s,\bullet) \ast I_{d-\beta} \ast \mu \in L^2(\mathbb{R}^d)$. Hence we obtain, by taking a Fourier transform and using Plancherel's theorem, that
\begin{align*}
& \int_{\RR^{d}} \varphi_R(t-s,y)\left[\varphi_R(t-s,\bullet) \ast I_{d-\beta} \ast \mu\right](y)dy \\
& \qquad = \frac {c_{d,\beta} } {(2\pi)^d}  \int_{\RR^d} \left|[\widehat{\varphi}_R(t-s,\bullet)](\xi)\right|^2 |\xi|^{\beta-d}\widehat{\mu}(\xi)d\xi,
\end{align*}
where $c_{d,\beta} =1$ for $\beta=d$.
By recalling that 
$$
\left| \int_{B_R} e^{-i\langle x,\xi\rangle}dx\right|^2 = (2\pi R)^d|\xi|^{-d}J_{\frac{d}{2}}^2(R|\xi|),
$$
where $J_{\frac d2}$ denotes the Bessel function of the first kind of order $d/2$, we obtain
\begin{equation*}
\left|[\widehat{\varphi}_R(t-s,\bullet)](\xi)\right|^2 = (2\pi R)^d|\xi|^{-d}J_{\frac{d}{2}}^2(R|\xi|) e^{-2(t-s)|\xi|^{\alpha}}
\end{equation*}
leading to 
\begin{eqnarray*}
&& \frac {c_{d,\beta}} {(2\pi)^d}\int_{\RR^d} \left|[\widehat{\varphi}_R(t-s,\bullet)](\xi)\right|^2 |\xi|^{\beta-d}\widehat{\mu}(\xi)d\xi \\
&=& c_{d,\beta}  \int_{\RR^d} R^d|\xi|^{-d}J_{\frac{d}{2}}^2(R|\xi|) e^{-2(t-s)|\xi|^{\alpha}} |\xi|^{\beta-d}\widehat{\mu}(\xi)d\xi \\
&=&  c_{d,\beta} R^{2d-\beta}\int_{\RR^d} |\xi|^{-d}J_{\frac{d}{2}}^2(|\xi|) e^{-2(t-s)R^{-\alpha}|\xi|^{\alpha}} |\xi|^{\beta-d}\widehat{\mu}\left(\frac{\xi}{R}\right)d\xi.
\end{eqnarray*}
Since $\widehat{\mu} \in L^{\infty}(\RR^d)$, we have $\sup_{R>0}e^{-2(t-s)R^{-\alpha}|\xi|^{\alpha}}\widehat{\mu}\left(\frac{\xi}{R}\right) < \infty$. Moreover, since $J_{\frac{d}{2}}^2(|\xi|) = O(|\xi|)$ as $|\xi|\to \infty$ and $J_{\frac{d}{2}}^2(|\xi|) \sim c_d|\xi|^d$ as $|\xi|\to 0$ (here we have used standard Landau notation $O(|\xi|)$ and $f\sim g$ if $\frac{f}{g} \to 1$), we have  
$
\int_{\RR^d} |\xi|^{\beta-2d}J_{\frac{d}{2}}^2(|\xi|)d\xi < \infty.
$
This, together with the boundedness of $\theta(s)$, allows us to use the dominated convergence theorem and therefore, as $R\to \infty$,
\begin{eqnarray*}
&& R^{\beta-2d}\int_0^t \theta(s) \int_{\RR^{2d}} \varphi_R(t-s,y)\varphi_R(t-s,y')\gamma(y-y')dy'dyds \\
&\to &  c_{d,\beta} \int_0^t \theta(s) ds \int_{\RR^d} |\xi|^{\beta-2d}J_{\frac{d}{2}}^2(|\xi|) \widehat{\mu}(0)d\xi.
\end{eqnarray*}
The result now follows from $\widehat{\mu}(0) = \mu(\RR^d)$ together with the fact that 
\begin{equation}
\label{eq:beta-riesz-constant}
c_{d,\beta} \int_{\RR^d} |\xi|^{\beta-2d}J_{\frac{d}{2}}^2(|\xi|) d\xi = \int_{B_1^2} |x_1-x_2|^{-\beta} dx_1dx_2
\end{equation}
for $\beta<d$ and, for $\beta=d$, we have
\begin{equation}
\label{eq:betad-constant}
\int_{\RR^d} |\xi|^{-d}J_{\frac{d}{2}}^2(|\xi|) d\xi = |B_1|.
\end{equation}
Indeed, the validity of \eqref{eq:betad-constant} can be seen from
\begin{equation*}
\int_{\RR^d}\textbf{1}_{B_1}(x) dx = \int_{\RR^d}\mathbf{1}^2_{B_1}(x) dx = \frac 1{(2\pi)^d}  \int_{\RR^d} \left|\widehat{\mathbf{1}_{B_1}}(\xi)\right|^2d\xi 
= \int_{\RR^d} |\xi|^{-d}J_{\frac{d}{2}}^2(|\xi|) d\xi
\end{equation*}
while the validity of \eqref{eq:beta-riesz-constant} can be seen from 
\begin{eqnarray*}
\int_{B_1^2}|x_1-x_2|^{-\beta} dx_1dx_2 &=& \int_{\RR^{d}} \mathbf{1}_{B_1}(x_1) \left[I_{d-\beta}\mathbf{1}_{B_1}\right](x_1) dx_1 \\
&=& \frac {c_{d,\beta}} { (2\pi)^{d}} \int_{\RR^d} \left|\widehat{\mathbf{1}_{B_1}}(\xi)\right|^2|\xi|^{\beta-d}d\xi \\
&=& c_{d,\beta}\int_{\RR^d} |\xi|^{\beta-2d}J_{\frac{d}{2}}^2(|\xi|) d\xi.
\end{eqnarray*}
This completes the proof.
\end{proof}

\begin{lemma}\label{sup}
Suppose that $\gamma$ satisfies Assumption \ref{assumption:main} and $\beta <d$. Then
\begin{equation*}
\lim_{\vert z\vert \to \infty }\sup_{0\leq s\leq t} \left| \Psi (s,  z) -\theta^2_{\alpha} (s) \right|=0.
\end{equation*}
\end{lemma}  

\begin{proof}
As in the proof of Theorem 3.1 in \cite {HNVZ19}, we can write, via the Clark-Ocone formula,
\begin{equation*}
\Psi (s,  y- y' ) - \theta^2_{\alpha} (s)= T(s,  y,  y' ), 
\end{equation*}
where 
\begin{equation*}
\vert T(s,  y,  y' )\vert \leq C \int_{0} ^ {s} \int_{\mathbb{R}^{2d} }    \Vert D_{r,  z} u(s,  y) \Vert _{2}  \Vert D_{r,  z'} u(s,  y') \Vert _{2}\gamma(z-z')dz'dzdr.
\end{equation*}
Hence, by applying Proposition \ref{pp11}, we obtain the estimate
\begin{eqnarray*}
\vert T(s,  y,  y' )\vert &\leq &C \int_{0} ^ {s}(s-r) ^{-\kappa}  \int_{\mathbb{R}^{2d} }     G_\alpha^{\frac 1{2q}} (s-r,  y- z) G_\alpha^{\frac 1{2q}} (s-r,  y'- z')\\
&&\times \gamma(z-z')dz'dzdr = : T_1(s,y,y').
\end{eqnarray*}
We prove the claim by an argument based on uniform integrability.
We know that  $\gamma = K_{d-\beta} \ast \mu$. Therefore,
\begin{eqnarray*}
T_1(s,  y,  y' )
   & =& C \int_{0} ^ {s}(s-r) ^{-\kappa}  \int_{ \R^{3d} }     G_\alpha^{\frac 1{2q}} (s-r,  y- z) 
G_\alpha^{\frac 1{2q}} ( s-r,  y'- z')\\
&&\times \vert  z- z'-w\vert ^{- \beta}dz'dz d\mu(w)dr,
\end{eqnarray*}
where $2q= \frac {2d}{2d-\beta}$.
Making the change of variables $u=s-r$, $\xi=y-z$ and $\xi'=y-z'$, we can write
\begin{eqnarray*}
T_1(s,  y,  y' )  & =& C \int_{0} ^ {s} u ^{-\kappa}  \int_{ \R^{3d} }     G_\alpha^{\frac 1{2q}} (u,  \xi) 
G_\alpha^{\frac 1{2q}} ( u,  \xi')\\
&&\times \vert  y-y'-\xi-\xi' -w\vert ^{- \beta}d\xi'd\xi d\mu(w)du.
\end{eqnarray*}
For any fixed $\xi, \xi' ,w \in \R^d$,  clearly,  $\vert  y-y'-\xi-\xi' -w\vert ^{- \beta}$ tends to zero as $|y-y'|$ tends to infinity.
Taking into account that
\[
 \int_{0} ^ {s} u ^{-\kappa}  \int_{ \R^{3d} }     G_\alpha^{\frac 1{2q}} (u,  \xi) 
G_\alpha^{\frac 1{2q}} ( u,  \xi') d\xi'd\xi d\mu(w)du <\infty,
\]
to show that  $\lim_{|y-y'| \rightarrow \infty} T_1(s,  y,  y' ) =0$, it suffices to check that
\[
I:= \int_{0} ^ {s} u ^{-\kappa}  \int_{\R^{3d}  }     G_\alpha^{\frac 1{2q}} (u,  \xi) 
G_\alpha^{\frac 1{2q}} ( u,  \xi') \vert  y-y'-\xi-\xi' -w\vert ^{- \beta'}d\xi'd\xi d\mu(w)du<\infty
\]
for some $\beta'>\beta$.  Making a change of variables, we can write
\[
I= \int_{0} ^ {s} u ^{-\kappa}  \int_{\R^{3d}  }     G_\alpha^{\frac 1{2q}} (u,  y-z) 
G_\alpha^{\frac 1{2q}} ( u,  y'-z') \vert   z-z' -w\vert ^{- \beta'}dz'dz d\mu(w)du.
\]
Appying H\"older's and   Hardy-Littlewood-Sobolev's inequality (\ref{eq:HLS-lp-mapping}) yields
\[
I \le C \int_{0} ^ {s} u ^{-\kappa} du  \left( \int_{\R^d} G_\alpha^{\frac {2d-\beta}{2d-\beta'}} (u,x) dx \right) ^{\frac {2d-\beta'} d},
\]
which is finite since $\beta'$ is close to $\beta$.
 This concludes the proof.
\end{proof}

\begin{proof}[Proof of Theorem \ref{pp2}] 
For notational simplicity, we only consider the case $r=t$ while the case of general $t,r\in [0,T]$ follows in a similar way. Using \eqref{mild} and \eqref{not1}, we can write 
$$
G_R(t) = \int_0^t \varphi_R(t-s,y)\sigma(u(s,y))W(ds,dy).
$$
Hence, by \eqref{inte}, we get
\begin{eqnarray*}
\E [G^2_{R}(t)  ]&=& \int_{0} ^ {t} \int_{\mathbb{R}^{2d}}  \varphi _{R}(t-s, y) \varphi_{R}(t-s, y' ) \Psi(s, y' -y)\gamma(y-y')dy'dyds.
\end{eqnarray*}
Let us begin with the case $\beta<d$. In view of Lemma \ref{pro:correct-limit} together with the boundedness of $\theta^2_{\alpha} (s)$, it suffices to show that 
\begin{eqnarray}
T_{R} &:=& R^{\beta - 2d} \int_0^t \int_{\mathbb{R}^{2d}}   \left[ \Psi (s, y-y ')-\theta^2_{\alpha} (s) \right] \varphi_{R} (t-s, y) \varphi_{R} (t-s, y') \nonumber\\
&& \phantom{kukkuu}\times   \gamma(y-y')dy'dyds \to 0.\label{L-0}
\end{eqnarray}
Now by  Lemma \ref{sup} we know that for every $\varepsilon >0$ there exists $K>0$ such that, for every $s\in [0, t]$ and every $y, y'$ with $ \vert y- y' \vert \geq K$,
\begin{equation}
\label{eq:small}
\left| \Psi(s, y-y') -\theta^2_{\alpha} (s) \right|\le \varepsilon.
\end{equation}
By using $\gamma = I_{d-\beta} \ast \mu$, we split  
$
T_{R}= T_{R,1}+ T_{R,2},
$
where
\begin{eqnarray*}
T_{R,1} &=& R^{\beta-2d} \int_{0}^{t}  \int_{\mathbb{R}^{3d} } \varphi_{R} (t-s, y) \varphi_{R} (t-s, y')    \left[\Psi (s, y-y ')- \theta^2 _{\alpha}(s) \right] \\
& & \phantom{kukkuu}\times \vert y-y'-w\vert^{-\beta} 1_{\vert y-y '\vert\leq K}d\mu(w) dy'dyds
\end{eqnarray*}
and
\begin{eqnarray*}
T_{R, 2} &=& R^{\beta-2d} \int_{0}^{t}  \int_{\mathbb{R}^{3d}}   \varphi_{R} (t-s, y) \varphi_{R} (t-s, y')    \left[\Psi (s, y-y ')- \theta^2_{\alpha} (s) \right] \\
& & \phantom{kukkuu}\times \vert y-y'-w\vert^{-\beta}1_{\vert y-y '\vert\geq K}d\mu(w)dy'dyds.
\end{eqnarray*}
On the region $|y'-y|\leq K, 0\leq s \leq T$ the quantity $\Psi (s, y-y ')- \theta^2_{\alpha} (s)$ is uniformly bounded. Using also the semigroup property and (\ref{density}) allows us to estimate
\begin{eqnarray*}
T_{R,1} &\le&CR ^{\beta -2d} \int_{0}^{t}  \int_{\mathbb{R}^{3d}} \int_{B_R^2}   G_{ \alpha } (t-s,  x- y) G_{ \alpha} (t-s,  x- y') \\
&&\phantom{kukkuu}\times \vert y-y '-w\vert ^{-\beta}1_{\vert y-y '\vert\leq K} dx'dxdy'dyd\mu(w)ds\\
&=& CR ^{\beta -2d}  \int_{0}^{t}  \int_{\mathbb{R}^{2d}}  \int_{B_R^2}  G_{ \alpha } (2(t-s),  x-x'- y')  
\vert y'-w\vert ^{-\beta}\\
&&\phantom{kukkuu}\times  1_{\vert y '\vert\leq K}dx'dxdy'd\mu(w)ds\\
&\leq &C R ^{\beta -d}\int_{\mathbb{R}^{2d}}  \vert y-w\vert ^{-\beta}1_{\vert y \vert\leq K}dyd\mu(w) \to 0
\end{eqnarray*}
as $R\to \infty$, since clearly here we have
$$
\int_{\mathbb{R}^{2d}}  \vert y-w\vert ^{-\beta}1_{\vert y \vert\leq K}dyd\mu(w) < \infty.
$$
For the term $T_{R,2}$, we apply \eqref{eq:small}  to get
\begin{eqnarray*}
T_{R, 2}
&\le&\varepsilon C_\alpha R ^{\beta -2d} \int_{0}^{t}  \int_{\mathbb{R}^{3d}}   \int_{B_R^2}  
G_{ \alpha} (t-s,  x- y) G_{ \alpha } (t-s,  x'- y') \\
&& \times \vert y- y '-w\vert ^{-\beta}dx'dxdy'dyd\mu(w)ds.
\end{eqnarray*}
The change of variables $x-y= \theta$, $x'-y' =\theta'$, $x_1=R\xi_1$ and $x'= R\xi'$ yields
\begin{align*}
T_{R, 2} &
\le  \varepsilon C_\alpha   \int_{0}^{t}  \int_{\mathbb{R}^{3d}}   \int_{B_1^2}  
G_{ \alpha} (t-s,  \theta) G_{ \alpha } (t-s,  \theta')  \\
& \qquad \times \vert   \xi - \xi' - R^{-1}\theta +R^{-1}\theta' -w\vert ^{-\beta}d\xi'd \xi d\theta d\theta' d\mu(w)ds,
\end{align*}
which is bounded by $C\varepsilon$ because $\sup_{z\in \R^d} \int_{B_1} |y-z| ^{-\beta} dy <\infty$.
 Since $\varepsilon>0$ is arbitrary, the desired limit (\ref{L-0}) follows. This verifies the claim for the case $\beta < d$.

Let next $\beta=d$. Since for a fixed $s>0$, the function $y \mapsto \Psi(s,y)$ is a bounded function and now $\gamma = \mu$ is a finite measure, we may regard $\tilde{\gamma}_s(dy) = \Psi(s,y)\gamma(dy)$ as another finite measure. Thus we may use exactly the same arguments as in the proof of Lemma  \ref{pro:correct-limit} and get
\begin{eqnarray*}
R^{-d}\E [G^2_{R}(t) ] &=& R^{-d}\int_{0} ^ {t} \int_{\mathbb{R}^{2d}}  \varphi _{R}(t-s, y) \varphi_{R}(t-s, y' ) \Psi(s, y' -y)\gamma(y-y')dy'dyds \\
&\to & \int_0^t \widehat{\tilde{\gamma}_s}(0)\int_{\RR^d} |\xi|^{-d}J_{\frac{d}{2}}^2(|\xi|) d\xi ds,
\end{eqnarray*}
where now 
$$
\widehat{\tilde{\gamma}_s}(0) = \int_{\RR^d} \Psi(s,z)d\gamma(z) = \nu_\alpha^2(s).
$$
This verifies the claim for $\beta=d$ as well, and hence the proof is completed.
\end{proof}
We end this subsection by proving Lemma \ref{lemma:nu-okay}.
\begin{proof}[Proof of Lemma \ref{lemma:nu-okay}]
Denote 
$$
T(s,y) := \Psi(s,y) - \theta^2_\alpha(s).
$$
Since $T(s,y)$ is also a bounded function, we may follow the proofs of Theorem \ref{pp2} and Lemma \ref{pro:correct-limit} to obtain 
\begin{eqnarray*}
&& R^{-d}\int_{0} ^ {t} \int_{\mathbb{R}^{2d}}  \varphi _{R}(t-s, y) \varphi_{R}(t-s, y' ) \Psi(s, y' -y)\gamma(y-y')dy'dyds \\
&=& R^{-d}\int_{0} ^ {t} \int_{\mathbb{R}^{2d}}  \varphi _{R}(t-s, y) \varphi_{R}(t-s, y' ) T(s, y' -y)\gamma(y-y')dy'dyds \\
&+& R^{-d}\int_{0} ^ {t} \theta^2_\alpha(s)\int_{\mathbb{R}^{2d}}  \varphi _{R}(t-s, y) \varphi_{R}(t-s, y' ) \gamma(y-y')dy'dyds,
\end{eqnarray*}
where now, as $R\to \infty$,
\begin{eqnarray*}
&&R^{-d}\int_{0} ^ {t} \theta^2_\alpha(s)\int_{\mathbb{R}^{2d}}  \varphi _{R}(t-s, y) \varphi_{R}(t-s, y' ) \gamma(y-y')dy'dyds\\
&\to & \mu\left(\RR^d\right)|B_1|\int_0^t \theta^2_\alpha(s) ds
\end{eqnarray*}
and 
\begin{eqnarray*}
&& R^{-d}\int_{0} ^ {t} \int_{\mathbb{R}^{2d}}  \varphi _{R}(t-s, y) \varphi_{R}(t-s, y' ) T(s, y' -y)\gamma(y-y')dy'dyds \\
&\to |B_1|& \int_0^t \widehat{T(s,\bullet)\gamma(\bullet)}(0) ds.
\end{eqnarray*}
By the very definition, we have
$$
T(s,y) := \Psi(s,y) - \theta^2_\alpha(s) = \text{Cov}\left[\sigma(u(s,y))\sigma(u(s,0))\right]
$$
and since $T(s,y)$ and $\gamma(y)$ are both covariances, they are positive semidefinite. Consequently, the product $T(s,y)\gamma(y)$ is again a covariance. It follows that 
$$
\widehat{T(s,\bullet)\gamma(\bullet)}(\xi) \geq 0
$$
for all $\xi \in \RR^d$ and, in particular, for $\xi = 0$. Now 
\begin{eqnarray*}
\int_0^t \int_{\RR^d} \Psi(s,z)d\mu(z)dt &=& \int_0^t \nu_\alpha^2(s)ds \\
&=& \int_0^t \widehat{T(s,\bullet)\gamma(\bullet)}(0)ds 
+ \int_0^t \theta_\alpha^2(s)ds.
\end{eqnarray*}
The claim follows from this together with the observations $\Psi(0,z)= \theta_\alpha^2(0)$ for all $z\in \RR^d$ and $\widehat{T(s,\bullet)\gamma(\bullet)}(0) \geq 0$ for all $s\geq 0$.
\end{proof}

\subsection{Proof of Theorem \ref{t1.2}}
\label{subsec:clt}

We start with the following result that we will utilise in the case $\beta<d$.
\begin{lemma}\label{2s-1}
Suppose that $0<\beta <\alpha  < 2\wedge d$.
For every $ t>0$ we have
\begin{equation}
\label{2s-2}
  \int_{\mathbb{R}^d}  G_{{\alpha} } (t, x-y) \vert y\vert ^{-\beta} dy   \leq C_{\beta,\alpha} |x| ^{-\beta}.
\end{equation}
\end{lemma}
\begin{proof}
Using the estimate \eqref{Eq1}, we have
\begin{align*}
&\int_{\mathbb{R}^d}  G_{{\alpha} } (t, x-y) \vert y\vert ^{-\beta} dy \\
&  \le C \int_{\R^d} \frac  { t^{- \frac d \alpha}  } { 1+  | (x-y) t^{-\frac 1 \alpha} |^{\alpha+d} }  | y| ^{-\beta} dy \\
&=  C \int_{|y| < \frac { |x|}2  } \frac  { t^{- \frac d \alpha}  } { 1+  | (x-y) t^{-\frac 1 \alpha} |^{\alpha+d} }  | y| ^{-\beta} dy 
+ C \int_{|y| \ge \frac { |x|} 2} \frac  { t^{- \frac d \alpha}  } { 1+  | (x-y) t^{-\frac 1 \alpha} |^{\alpha+d} }  | y| ^{-\beta} dy   \\
& \le   C \int_{|y| < \frac { |x|}2  } \frac  { t^{- \frac d \alpha}  } { 1+  | x   t^{-\frac 1 \alpha} |^{\alpha+d} }  | y| ^{-\beta} dy 
+ C | x| ^{-\beta}   \int_{|y| \ge \frac { |x|} 2} \frac  { t^{- \frac d \alpha}  } { 1+  | (x-y) t^{-\frac 1 \alpha} |^{\alpha+d} }  dy \\
& \le   C  \frac  { t^{- \frac d \alpha}  } { 1+  | x   t^{-\frac 1 \alpha} |^{\alpha+d} }  | x| ^{-\beta+d} 
+ C | x| ^{-\beta} .
\end{align*}
The estimate (\ref{2s-2}) follows from this, because one can show that
\[
\sup_{x\in \R^d} \sup_{t>0} \frac  { t^{- \frac d \alpha}  |x|^d } { 1+  | x   t^{-\frac 1 \alpha} |^{\alpha+d} } <\infty.
\]
\end{proof}
\begin{proof}[Proof of Theorem \ref{t1.2}]
Let $\varphi_R$ be given by \eqref{not1}. By the same arguments as in the proof of \cite[Theorem 1.1]{HNVZ19}, 
using Proposition \ref{lem:dist},  Theorem \ref{pp2} and  Proposition \ref{pp11},
we get $d_{TV}(F_R,Z)\le  2(A_1+A_2)$, where  
\begin{eqnarray}
A_1&\leq & CR^{\beta-2d}\int_0^t\left(\int_{0}^s   (s-r) ^ {-\kappa}\int_{\mathbb{R}^{6d}} \varphi_{R} (t-s, y)\varphi_{R} (t-s,  y')\varphi_{R} (t-s, \tilde  y)
\right. \nonumber\\
&& \times  \varphi_{R} (t-s, \tilde  y') 
  G^{\frac 1{2q}}_{ \alpha } (s-r,  y- z)  G^{\frac 1{2q}}_{ \alpha } (s-r,\tilde{  y}- z') \gamma(y-y')\nonumber \\
&&\left.\times \gamma(\tilde  y-\tilde y')\gamma( z- z')d yd y'd\tilde yd\tilde  y'd zd z'dr \right)^{1/2}ds \label{eq:A1-bound}
\end{eqnarray}
and
\begin{eqnarray*}
A_{2} &\leq & C R^{\beta-2d}\int_{0} ^{t}  \bigg( \int_{s}^{t}  (r-s) ^{-\kappa}\int_{ \mathbb{R}^{6d}} \varphi_{R} (t-r,z) \varphi_{R} (t-r, \tilde{z}) \varphi_{R} (t-s,y')  \\
&&\times \varphi_{R} (t-s, \tilde{y} ') G^{\frac 1{2q}}_{ \alpha} (r-s, z- y)  G^{\frac 1{2q}}_{ \alpha} (r-s,\tilde{z}- \tilde{y} )  \\
&&\times \gamma( y- y')\gamma(\tilde  y-\tilde  y')\gamma(  z-\tilde  z)d y d y'd\tilde  yd\tilde  y'd zd\tilde  zdr\bigg)^{1/2}ds. 
\end{eqnarray*}
We begin with the case $\beta=d$ that is simpler. For the term $A_1$ in this case, we use the trivial bound $\varphi_{R} (t-s,  y')\varphi_{R} (t-s, \tilde  y)\varphi_{R} (t-s, \tilde  y') \leq 1$, integrate in the variables $y'$ and $\tilde{y}'$, perform the  change of variables $y\mapsto y-z$ and $\tilde{y} \mapsto \tilde{y}-z$ in the integrals with respect to $y,\tilde{y}$, and then integrate with respect to $z'$, $z$, and finally with respect to $y$ and $\tilde{y}'$. Together with Lemma \ref{lma:2q-integral}, this leads to 
\begin{eqnarray*}
A_1&\leq & CR^{-d}\int_0^t\left(\int_{0}^s   (s-r) ^ {-\kappa}\int_{\mathbb{R}^{4d}} \varphi_{R} (t-s, y)G^{\frac{1}{2q}}_{ \alpha } (s-r,  y- z)  \right. \nonumber\\
&& \left.\times  G^{\frac{1}{2q}}_{ \alpha } (s-r,\tilde{  y}- z') \gamma( z- z')d yd\tilde yd zd z'dr \right)^{1/2}ds\nonumber \\
&= & CR^{-d}\int_0^t\left(\int_{0}^s   (s-r) ^ {-\kappa}\int_{\mathbb{R}^{4d}} \varphi_{R} (t-s, y+z)G^{\frac{1}{2q}}_{ \alpha } (s-r,  y)  \right. \nonumber\\
&& \left.\times  G^{\frac{1}{2q}}_{ \alpha } (s-r,\tilde{  y}) \gamma( z- z')d zd z'd yd\tilde ydr \right)^{1/2}ds\nonumber \\
&\leq &CR^{-\frac{d}{2}}\int_0^t\left(\int_{0}^s   (s-r) ^ {-\kappa}\int_{\mathbb{R}^{4d}} G^{\frac{1}{2q}}_{ \alpha } (s-r,  y) G^{\frac{1}{2q}}_{ \alpha } (s-r,\tilde{  y})d yd\tilde ydr \right)^{1/2}ds\nonumber \\
&\leq & CR^{-\frac{d}{2}}.
\end{eqnarray*}
Treating the term $A_2$ with similar arguments completes the proof for the case $\beta=d$. Suppose next $\beta<d$ and let us again first treat the term $A_1$. We can bound $A_1$ as  follows
\begin{align*}
A_1 & \le
C R^{\beta -2d} \int_0^t  \Bigg( \int_0^s (s-r)^{-\kappa}  
\int_{B_R^4} \int_{\R^{6d}} G_\alpha(t-s, x_1-y) G_\alpha(t-s, x_2-y') \\
& \qquad \times G_\alpha(t-s, x_3- \tilde{y})G_\alpha(t-s, x_4-\tilde{y}')  G^{\frac 1{2q}}_{ \alpha} (r-s, z- y)  G^{\frac 1{2q}}_{ \alpha} (r-s,\tilde{z}- \tilde{y} ) \\
& \qquad \times \vert  y- y'-w_1\vert^{- \beta}\vert \tilde  y-\tilde  y'-w_2\vert^{- \beta}\vert  z-\tilde  z-w_3\vert^{- \beta}\\
& \qquad \times d y d y'd\tilde  yd\tilde  y'd zd\tilde  z d\mu(w_1)d\mu(w_2)d\mu(w_3)dr\bigg)^{1/2}ds.
\end{align*}
The change of variables $x_1-y=\theta_1$, $ x_2-y' =\theta_2$, $x_3 -\tilde{y} =\theta_3$, $x_4-\tilde{y}' =\theta_4$, $z-y= \eta_1$ and $\tilde{z} - \tilde{y} = \eta_2$,  yields
\begin{align*}
A_1 & \le
C R^{\beta -2d} \int_0^t  \Bigg( \int_0^s (s-r)^{-\kappa}  
\int_{B_R^4} \int_{\R^{6d}} G_\alpha(t-s, \theta_1) G_\alpha(t-s, \theta_2) \\
& \qquad \times G_\alpha(t-s, \theta_3)G_\alpha(t-s, \theta_4) G^{\frac 1{2q}}_{ \alpha} (r-s, \eta_1)  G^{\frac 1{2q}}_{ \alpha} (r-s, \eta_2 ) \\
& \qquad \times \vert  x_1-x_2+\theta_2-\theta_1-w_1\vert^{- \beta}\vert  x_3 -x_4 +\theta_4-\theta_3-w_2\vert^{- \beta}\\
& \qquad \times \vert x_1-x_3 -\theta_1+\theta_4 +\eta_1-\eta_2-w_3\vert^{- \beta}\\
&\qquad \times d \theta_1 d \theta_2 d\theta_3   d\theta_4 d\eta_1d\eta_2d\mu(w_1)d\mu(w_2)d\mu(w_3)dr\bigg)^{1/2}ds.
\end{align*}
Integrating in the variables $\theta_2$ and $\theta_3$ and using the estimate (\ref{2s-2}), we can write
\begin{align*}
A_1 & \le
C R^{\beta -2d} \int_0^t  \Bigg( \int_0^s (s-r)^{-\kappa}  
\int_{B_R^4} \int_{\R^{4d}} G_\alpha(t-s, \theta_1)   \\
& \qquad \times  G_\alpha(t-s, \theta_4) G^{\frac 1{2q}}_{ \alpha} (r-s, \eta_1)  G^{\frac 1{2q}}_{ \alpha} (r-s, \eta_2 ) \\
& \qquad \times \vert  x_1-x_2-\theta_1-w_1\vert^{- \beta}\vert  x_3 -x_4 +\theta_4-w_2\vert^{- \beta}\\
& \qquad \times \vert x_1-x_3 -\theta_1+\theta_4 +\eta_1-\eta_2-w_3\vert^{- \beta}\\
& \qquad \times d \theta_1   d\theta_4 d\eta_1d\eta_2d\mu(w_1)d\mu(w_2)d\mu(w_3)dr\bigg)^{1/2}ds.
\end{align*}
The change of variables $x_i =R \xi_i $, $i=1,2,3,4$ yields
\begin{align*}
A_1 & \le
C R^{-\beta/2} \int_0^t  \Bigg( \int_0^s (s-r)^{-\kappa}  
\int_{B_1^4} \int_{\R^{4d}} G_\alpha(t-s, \theta_1)   \\
& \qquad \times  G_\alpha(t-s, \theta_4) G^{\frac 1{2q}}_{ \alpha} (r-s, \eta_1)  G^{\frac 1{2q}}_{ \alpha} (r-s, \eta_2 ) \\
& \qquad \times \vert  x_1-x_2-R^{-1}[\theta_1-w_1]\vert^{- \beta}\vert  x_3 -x_4 +R^{-1}[\theta_4-w_2]\vert^{- \beta}\\
& \qquad \times \vert x_1-x_3 +R^{-1} [-\theta_1+\theta_4 +\eta_1-\eta_2-w_3]\vert^{- \beta}\\
&\qquad \times d \theta_1   d\theta_4 d\eta_1d\eta_2d\mu(w_1)d\mu(w_2)d\mu(w_3)dr\bigg)^{1/2}ds.
\end{align*}
Taking into account that
\[
\sup_{z\in \R^d} \int_{B_1} |x +z|^{-\beta} dx <\infty,
\]
and that, by Lemma \ref{lma:2q-integral},
\[
\int_{\R^d} G^{\frac 1{2q}}_{ \alpha} (r-s, \eta) d\eta = C (r-s) ^{\frac{\kappa}{2}},
\]
we conclude that
$$
A_1 \leq CR^{-\beta/2}.
$$
Treating the term $A_2$ similarly verifies the case $\beta<d$ as well, completing the whole proof.
\end{proof}

\subsection{Proof of Theorem \ref{t1.4}}
\label{subsec:fclt}
In order to prove \ref{t1.4} it suffices to prove tightness and the convergence of the finite dimensional distributions. For the latter we can proceed as in \cite{HNVZ19} together with the arguments of the proof of Theorem \ref{t1.2}. The tightness in ensured by the following result and Kolmogorov's criterion.

\begin{proposition}
\label{pro:tightness2}
Let $u(t,x)$ be the solution to \eqref{1}. Then for any $0\leq s < t\leq T$ and any $p\geq 1$ there exists a constant $C=C(p,T)$ such that 
$$
\E \left(\left|\int_{B_R} u(t,x)dx - \int_{B_R}u(s,x)dx\right|^p \right) \leq CR^{\left(d-\frac{\beta}{2}\right)p}(t-s)^{\frac{p}{2}}.
$$
\end{proposition}
\begin{proof}
Let $ \Theta_{x,t,s}$ be given by 
$$
\Theta_{x,t,s}(r,y) = G_{\alpha}(t-r,x-y)\1_{\{r\leq t\}} - G_{\alpha}(s-r,x-y)\1_{\{r\leq s\}}.
$$
We have, for $0<s<t<T$,
\begin{equation*}
\int_{B_R}u(t,x)dx - \int_{B_R}u(s,x)dx=
\int_0^T  \int_{\R^d} \int_{B_R} \Theta_{x,t,s}(r,y) \sigma(u(r,y))dxW(dr,dy).
\end{equation*}
Now Burkholder inequality implies that, for every $p\geq 1$, 
\begin{align*}
&\E\left(\left|\int_{B_R} u(t,x)dx - \int_{B_R} u(s,x)dx\right|^p \right)\\
\leq & C_{p,T}\left(\int_0^T \int_{\R^{2d}}    \left(\int_{B_R^2}   \Theta_{x,t,s}(r,y) \Theta_{x',t,s}(r,y')dx'dx\right)\gamma(y-y') dy dy'dr \right)^{\frac{p}{2}}.
\end{align*}
Hence it remains to show that
\begin{eqnarray} 
K_{R}(t,s)&:=&\int_0^T \int_{\R^{2d}}     \left(\int_{B_R^2}  \Theta_{x,t,s}(r,y) \Theta_{x',t,s}(r,y')dx'dx\right) \gamma(y-y')dy dy'dr \nonumber\\
&\leq & C R^{2d-\beta} (t-s)\label{8a-4}.
\end{eqnarray}
By taking the Fourier transform, we obtain 
$K_{R}(t,s)\leq C (I_{1}+ I_{2}),$
where
$$I_{1}=  \int_0^s \int_{\RR^d} R^d|\xi|^{-d}J_{\frac{d}{2}}^2(R|\xi|)\left|e ^{-(t-r)\vert \xi \vert ^{\alpha}} - e ^{-(s-r)\vert \xi \vert ^{\alpha}}\right|^2 \widehat{\gamma}(\xi)d\xi dr$$
and
$$I_{2}= \int_s^t \int_{\RR^d} R^d|\xi|^{-d}J_{\frac{d}{2}}^2(R|\xi|)e ^{-2(t-r)\vert \xi \vert ^{\alpha}} \widehat{\gamma}(\xi)d\xi dr.$$
Using $e ^{-2(t-r)\vert \xi \vert ^{\alpha}}\leq 1$ and
$$
\left|e ^{-(t-r)\vert \xi \vert ^{\alpha}} - e ^{-(s-r)\vert \xi \vert ^{\alpha}}\right|^2 \leq C(t-s)
$$
leads to
\begin{eqnarray*}
I_1 + I_2 &\leq &C(t-s) \int_{\RR^d} R^d|\xi|^{-d}J_{\frac{d}{2}}^2(R|\xi|) \widehat{\gamma}(\xi)d\xi \\
& = & C(t-s)R^{2d-\beta}\int_{\RR^d} |\xi|^{\beta-2d}J_{\frac{d}{2}}^2(|\xi|) d\xi.
\end{eqnarray*}
This concludes the proof.
\end{proof}

Theorem \ref{t1.4} follows by the arguments of the proof of \cite[Theorem 1.3]{HNVZ19} together with Proposition \ref{pro:tightness2}. Details that, despite being rather lengthy, are directly based on the same arguments that we have used above, and for this reason they are left to the reader.

\section{Appendix: proof of Lemma \ref{ll2}}

\begin{proof}[Proof of Lemma \ref{ll2}]
As in \cite{CH1} (see the proofs of Lemmas 2.4 and 3.1), it suffices to prove the bound (\ref{20i-2}) in the case when (\ref{20i-1}) is an equality.  Let $g_{n}, n\geq 0$ be a sequence defined iteratively by setting $g_{0} (t,x)= G_{\alpha} (t, x)$, and for $n\geq 0$ 
\begin{eqnarray*}
g_{n+1}(t, x) ^{2}&=& G_{\alpha} (t,x) ^ {2} + \int_{0} ^ {t} \int_{\mathbb{R}^{2d}  }     G_{\alpha }(t-s, x-y )  G_{\alpha }(t-s, x-y' )\\
&&\times g_{n}(s, y) g_{n}(s, y' )  \gamma(y-y')dy'dyds .
\end{eqnarray*}
Denote $\kappa = \frac{2d}{\alpha} - \frac{d}{q\alpha}$. We prove by induction that for every $n\geq 0$,
\begin{equation}\label{20i-3}
g_{n}(t, x)^ {2}  \leq C\sum_{j=0}^ {n} \frac{ \Gamma^{j}(1-\kappa)}{\Gamma \big(  (j+1)(1-\kappa)\big)}  t ^{ j(1-\kappa) -\kappa}G^{\frac{1}{q}} _{\alpha} (t, x).
\end{equation}
For $n=0$, taking into account that  $\alpha +d \ge \frac { \alpha+ d}{2q}$, $\frac{\kappa}{2} = \frac{d}{\alpha}- \frac{d}{2q\alpha}$, and $2q>1$, we can use the estimate \eqref{Eq1},
\begin{eqnarray}  \notag
g_{0} (t, x)&= & G  _{\alpha } (t, x)  \leq C  \frac { t^{-\frac d\alpha} }  { (1+ | t^{-\frac 1\alpha} x|)^{ \alpha+d}} \\ \notag
&\le &C  \frac { t^{-\frac d\alpha}} { (1+ | t^{-\frac 1\alpha} x|)^{ (\alpha+d)/2q   }}\\  \label{ECU2}
&\le &C 
 t ^{-\frac{\kappa}{2}}  G _{\alpha}^{\frac 1{2q}}  (t, x).
\end{eqnarray}
Hence (\ref{20i-3}) is true for $n=0$.  

Suppose that (\ref{20i-3}) holds for $n$.
Denoting $c_j = \frac{ \Gamma^{j}(1-\kappa)}{\Gamma \big(  (j+1)(1-\kappa)\big)} $ and by the induction hypothesis, 
 \begin{eqnarray} \notag
g_{n+1}(t, x) ^ {2}&\leq& G_{\alpha } (t,x) ^ {2} + \int_{0} ^ {t}  \int_{\mathbb{R}^{2d} }   G_{\alpha }(t-s, x-y ) 
G_{\alpha }(t-s, x-y' )\\  \notag
&&\times   \sum_{j=0}^n  c_j s^{ j(1-\kappa) -\kappa} G_{\alpha}^{\frac 1{2q}}(s, y)G_{\alpha}^{\frac 1{2q}}(s, y')
\gamma(y-y')dy'dyds\\  \label{ECU3}
&=:& G_{\alpha } (t,x) ^ {2} +\sum_{j=0}^n  c_j  I_j.
\end{eqnarray}
The inequality \eqref{eq:key-convolution-bound} with $g(y)=f(y)= G_{\alpha}(t-s,x-y)G_{\alpha}^{\frac 1{2q}}(s, y)$ allows us to estimate
\[
I_j  \le
 C  \int_0^t s^ {j(1 -\kappa)-\kappa}  \left(\int_{\R^d}  G^{2q} _{\alpha }(t-s, x-y ) G_{\alpha} (s, y) dy \right)^{\frac 1q}ds.
\]
The scaling and asymptotic properties of the kernel   $G _{\alpha }$ imply that
\[
 G^{2q} _{\alpha }(t-s, x-y ) \le C
 \frac { (t-s)^{-\frac { 2qd} \alpha}}   { (1+ |(t-s)^{-\frac 1 \alpha} (x-y)|) ^{2q(\alpha+d)}}.
  \]
 Taking into account that  $(\alpha +d)2q \ge \alpha +d$, we obtain
 \begin{align*}
  G^{2q} _{\alpha }(t-s, x-y )& \le C
 \frac { (t-s)^{-\frac { 2qd} \alpha}}   { (1+ |(t-s)^{-\frac 1 \alpha} (x-y)|) ^{\alpha+d}} \\
 & \le  C  (t-s)^{- \kappa q}  G_{\alpha}( t-s, x-y).
 \end{align*}
 Therefore, by the semigroup property
 \begin{align}  \notag
 I_j & \le C \int_0^t s^{ j(1 -\kappa)-\kappa}  (t-s) ^{-\kappa}  G_{\alpha}^{\frac 1q}(t-s + s , x) ds \\ \notag
 & = C G_{\alpha} ^{\frac 1q}( t , x)  \int_0^t s^{ j(1 -\kappa)-\kappa}  (t-s) ^{-\kappa} ds  \\  \label{ECU1}
 &= C G_{\alpha} ^{\frac 1q}( t , x)  t^{ ( j+1)(1 -\kappa  ) -\kappa} 
 \frac { \Gamma(1-\kappa) \Gamma( (j+1) (1-\kappa ))}{\Gamma( (j+2)(1-\kappa) )}.
 \end{align}
 Substituting  (\ref{ECU1}) and  (\ref{ECU2}) into (\ref{ECU3}) yields
 \begin{align*}
 g_{n+1} (t,x)^2 & \le C  t^{-\kappa} G^{\frac 1q}_{\alpha} (t,x) \\
 & \qquad + C  \sum_{j=0}^n c_j   G_{\alpha} ^{\frac 1q}( t , x)  t^{ ( j+1)(1 -\kappa  ) -\kappa} 
 \frac { \Gamma(1-\kappa) \Gamma( (j+1) (1-\kappa ))}{\Gamma( (j+2)(1-\kappa) )}\\
 &= C  G_{\alpha} ^{\frac 1q}( t , x) 
   \sum_{j=0}^{n+1} c_{j-1}    t^{ j(1 -\kappa  ) -\kappa} 
 \frac { \Gamma(1-\kappa) \Gamma( j (1-\kappa ))}{\Gamma( (j+1)(1-\kappa) )} \\
 &= C  G_{\alpha} ^{\frac 1q}( t , x) 
   \sum_{j=0}^{n+1}     t^{ j(1 -\kappa  ) -\kappa} 
 \frac { \Gamma^j(1-\kappa)  }{\Gamma( (j+1)(1-\kappa) )} .
 \end{align*}

Finally, it follows from (\ref{20i-3}) 
\begin{eqnarray*}
g(t,x)&=& \lim _{n\to \infty} g_{n}(t,x) \leq C \left( \sum_{j=0}^ {\infty} \frac{\ \Gamma^{j}(1-\kappa )}{\Gamma \big(   (j+1)(1-\kappa)\big)} t ^ {j(1-\kappa)- \kappa}G _{\alpha } (t, x) \right)^{1/2}\\
&\leq & C t^ {-\frac \kappa{2}}  G^{\frac 1{2q}}_{\alpha} (t, x).
\end{eqnarray*}
This finishes the proof.
\end{proof}


\begin{thebibliography}{99}



\bibitem{Bogachev}
V. Bogachev. \emph{Measure Theory}. Spinrger-Verlag, Berlin, 2007.



\bibitem{CD}
L. Chen and R. Dalang. Moments, intermittency and growth indices for nonlinear stochastic fractional heat equation. \emph{Stoch. Partial Differ. Equ. Anal. Comput.}, 3(3): 360-397, 2015. 


\bibitem{CH1}
L. Chen and J. Huang. Comparison principle for stochastic heat equation on $\mathbb{R} ^ {d}$. \emph{Ann. Probab.}, 47(2): 989-1035, 2018.

\bibitem{CK}
L. Chen and K. Kunwoo. Nonlinear Stochastic Heat Equation Driven by Spatially Colored Noise: Moments and Intermittency. \emph{Acta Mathematica Scientia}, 39: 645-668, 2019.

 \bibitem{CHN}
  L. Chen, Y. Hu and D. Nualart. Regularity and strict positivity of densities for the nonlinear stochastic heat equation. To appear in: \emph{Mem. Amer. Math. Soc.}, 2018.
  
 \bibitem{Da} 
 R. Dalang. Extending the Martingale Measure Stochastic Integral With Applications to Spatially Homogeneous S.P.D.E.'s. 
  \emph{ Electron. J. Probab.} Volume 4, paper no. 6, 29 pp., 1999.
  
  \bibitem{DD}
L. Debbi and M. Dozzi. On the solutions of nonlinear stochastic
fractional partial differential equations in one spatial dimension. 
\emph{Stoch. Proc. Appl.}, 115: 1761-1781, 2005.

 \bibitem{DNZ18}
 F. Delgado-Vences, D. Nualart and G. Zheng.
 A central limit theorem for the stochastic wave equation with fractional noise. {\it ArXiv preprint: 1812.05019}, 2018.
 
\bibitem{Garofalo}
N. Garofalo.
 \emph{Fractional thoughts.}
 In: New Developments in the Analysis of Nonlocal Operators, Contemp. Math., 723: 1-135, Amer. Math. Soc., Providence, RI, 2019.
 
 
 
 \bibitem{GT}  
 B. Gaveau and P. Trauber. L'int\'egrale stochastique comme op\'erateur de divergence dans l'espace founctionnel. \emph{J. Funct. Anal.}, 46: 230-238, 1982.




\bibitem{HNV18}
J. Huang, D. Nualart and L.  Viitasaari.  A central limit theorem for the stochastic heat equation. \emph{ArXiv preprint: 1810.09492}, 2018.

\bibitem{HNVZ19}
J. Huang, D. Nualart, L.  Viitasaari and G. Zheng.  Gaussian fluctuations for the stochastic heat equation with colored noise. \emph{Stoch. Partial Differ. Equ. Anal. Comput.}, https://doi.org/10.1007/s40072-019-00149-3, 2019.

\bibitem{Ko}
T. Komatsu. On the martingale problem for generators of stable
processes with perturbations. \emph{Osaka. J. Math.}, 21: 113-132, 1984.

\bibitem{Lieb}
E.H. Lieb. Sharp constants in the Hardy-Littlewood-Sobolev and related inequalities. \emph{Ann. of Math.}, 118: 349-374, 1983.


\bibitem{NP09}
I.~Nourdin and G.~Peccati.
\newblock Stein's method on Wiener chaos.
\newblock {\em Probab. Theory Rel.}, 145(1):75-118, 2009.


\bibitem{NP} 
I. Nourdin and G. Peccati. {\it Normal approximations with Malliavin calculus.} 
From Stein's method to universality. Cambridge Tracts in Mathematics, 192. Cambridge University Press, Cambridge, 2012. xiv+239 pp.



\bibitem{Eulalia}
D. Nualart  and E.  Nualart.  \emph{Introduction to Malliavin Calculus}. IMS Textbooks, Cambridge University Press, 2018.


\bibitem{NuPa}  D. Nualart and E. Pardoux. Stochastic calculus with anticipating integrands. {\it Probab. Theor. Rel.}  78: 535-581, 1988.



\bibitem{Zhou}
D. Nualart  and H. Zhou. Total variation estimates in the Breuer-Major theorem. \emph{ArXiv preprint: 1807.09707}, 2018.


 
 

\bibitem{Walsh}  J. B. Walsh.
	\newblock {\it An Introduction to Stochastic Partial Differential Equations.}
	\newblock In: \'Ecole d'\'et\'e de probabilit\'es de
	Saint-Flour, XIV---1984, 265-439.
	Lecture Notes in Math.\ 1180, Springer, Berlin, 1986.



\end{thebibliography}
\end{document}